\documentclass[11pt]{article}
\usepackage{latexsym,amsmath} 
\usepackage{times}

\def\vec#1{\mathchoice{\mbox{\boldmath$\displaystyle#1$}}
{\mbox{\boldmath$\textstyle#1$}}
{\mbox{\boldmath$\scriptstyle#1$}}
{\mbox{\boldmath$\scriptscriptstyle#1$}}}

\newcommand{\qed}{\hfill$\Box$\smallskip}
\newenvironment{proof}{\emph{Proof.}}{}

\newtheorem{definition}{Definition}[section]

\newtheorem{remark}[definition]{Remark}
\newtheorem{theorem}[definition]{Theorem}
\newtheorem{lemma}[definition]{Lemma}
\newtheorem{proposition}[definition]{Proposition}
\newtheorem{corollary}[definition]{Corollary}

\newtheorem{fact}[definition]{Fact}

\newtheorem{conjecture}[definition]{Conjecture}

\usepackage[dvips]{epsfig} 
\graphicspath{{./Fig_eps/}{./Eps/}} 
\DeclareGraphicsExtensions{.ps,.eps} 
\usepackage{color} 
\usepackage{fullpage}

\newcommand\dist{\mbox{dist}}

\newcommand\cA{\mathcal{A}} 
\newcommand\cB{\mathcal{B}} 
\newcommand\cC{\mathcal{C}} 
\newcommand\cD{\mathcal{D}} 
\newcommand\cF{\mathcal{F}} 
 
\newcommand\cE{\mathcal{E}}

\newcommand\cH{\mathcal{H}} 
\newcommand\cS{\mathcal{S}} 
\newcommand\cT{\mathcal{T}}

\newcommand\cM{\mathcal{M}}

\newcommand\cZ{\mathcal{Z}}

\def\cC{{\mathcal C}}
\def\cE{{\cal E}}

\newcommand\eul{\mathrm{e}} 
\newcommand\eps{\varepsilon}

\newcommand\Erw{\mathrm{E}} 
\newcommand\pr{\mathrm{P}}

\newcommand{\vecone}{\vec{1}}

\newcommand{\Bin}{{\rm Bin}}

\newcommand{\bink}[2] {{{#1}\choose {#2}}}

\newcommand\ra{\rightarrow} 

\newcommand\bc[1]{\left({#1}\right)} 
\newcommand\cbc[1]{\left\{{#1}\right\}} 
\newcommand\bcfr[2]{\bc{\frac{#1}{#2}}} 
 
\newcommand\brk[1]{\left\lbrack{#1}\right\rbrack}

\newcommand\abs[1]{\left|{#1}\right|}

\newcommand\RR{\mathbf{R}} 
 
\newcommand{\Whp}{W.h.p.} 
\newcommand{\whp}{w.h.p.}

\newcommand\Lem{Lemma}
\newcommand\Prop{Proposition}
\newcommand\Thm{Theorem}
\newcommand\Cor{Corollary}
\newcommand\Sec{Section}

\begin{document} 

\title{The condensation transition in random hypergraph 2-coloring}

\author{
Amin Coja-Oghlan\thanks{University of Warwick, Zeeman building, Coventry CV4 7AL, UK,
	{\tt a.coja-oghlan@warwick.ac.uk}. Supported by EPSRC grant EP/G039070/2.}\ \ and
	Lenka Zdeborova\thanks{Institut de Physique Th\'eorique, IPhT, CEA Saclay, and URA 2306, CNRS, 91191 Gif-sur-Yvette, France,
		{\tt lenka.zdeborova@gmail.com}.
		}\\
} 
\date{\today}

\maketitle 

\begin{abstract}
\noindent
For many random constraint satisfaction problems such as random satisfiability or random graph or hypergraph coloring,
the best current estimates of the threshold for the existence of solutions are based on the \emph{first} and the \emph{second moment method}.
However, in most cases 
these techniques do not yield matching upper and lower bounds.
Sophisticated but non-rigorous arguments from statistical mechanics have ascribed this discrepancy to the
existence of a phase transition called \emph{condensation} that occurs shortly before the actual threshold
for the existence of solutions and that affects the combinatorial nature of the problem, rendering
the second moment method powerless
	(Krzakala, Montanari, Ricci-Tersenghi, Semerjian, Zdeborova: PNAS 2007).
In this paper we prove for the first time \emph{that} a condensation transition exists in 
a natural random CSP, namely  in random hypergraph $2$-coloring.
Perhaps surprisingly, we find that the second moment method breaks down strictly \emph{before} the condensation transition.
Our proof also yields slightly improved bounds on the threshold for random hypergraph $2$-colorability.
We expect that our techniques can be extended to other, related problems such as random $k$-SAT or random graph $k$-coloring.

\end{abstract}

\newpage
\section{Introduction and results}\label{Sec_Intro}

For many random constraint satisfaction problems such as random $k$-SAT, random graph coloring,
or random hypergraph coloring the best current bounds on the thresholds for the existence
of solutions derive from the \emph{first} and the \emph{second moment method}.
However, in most cases these simple techniques do not yield matching upper and lower bounds.
In effect, for most random CSPs the \emph{precise} threshold for the existence of solutions remains unknown.
Examples of this include random $k$-SAT, random graph $k$-coloring, and the $2$-coloring problem in
random $k$-uniform hypergraphs ($k\geq3$).

In this paper we investigate the origin of this discrepancy with the example of the random hypergraph 2-coloring
problem, in which the second moment analysis is technically relatively simple.
First, we present an approach to improve slightly over the naive second moment argument.
But more importantly, we establish the existence of a further phase transition below the threshold for the
existence of solutions. 
At this so-called \emph{condensation transition}, whose existence 
was predicted on grounds of sophisticated but non-rigorous statistical mechanics arguments~\cite{DRZ08,pnas},
the combinatorial nature of a `typical' solution becomes significantly more complicated.
Arguably, beyond the condensation transition it is conceptually more difficult to
prove that solutions exist, and indeed in several random CSPs condensation seems to pose the key obstacle to determining the precise threshold
for the existence of solutions.
Here we prove rigorously for the first time that a
condensation transition indeed exists. 

To define the \emph{random hypergraph $2$-coloring} problem,
let $V=\cbc{1,\ldots,n}$ be a set of vertices, let $k\geq3$, and let
$H_k(n,m)$ be a random $k$-uniform hypergraph on $V$ obtained by inserting
a random set of $m$ edges out of the $\bink{n}k$ possible edges.
A \emph{$2$-coloring} of $H$ is a map $\sigma:V\ra\cbc{0,1}$ such that
no edge $e$ of $H$ is monochromatic.
Throughout the paper, we will let $r=m/n$ denote the \emph{density} of the random hypergraph.
An event $\cE$ occurs \emph{with high probability} (`\whp')
if its probability  tends to one as $n\ra\infty.$
We let $\cS(H)$ denote the set of all $2$-colorings of the hypergraph $H$,
and we let $Z(H)=\abs{\cS(H)}$.

Friedgut's sharp threshold theorem implies that for any $k\geq3$ there \emph{exists} a  threshold
$r_{col}=r_{col}(k,n)$ such that for any $\eps>0$  the random hypergraph $H_k(n,m)$ of density
$r=m/n<(1-\eps)r_{col}$ is $2$-colorable \whp, while for $r>(1+\eps)r_{col}$  it is \whp\ not~\cite{Ehud,EhudHunting}.
	\footnote{It is widely conjectured that $\lim_{n\ra\infty}r_{col}(k,n)$ exists for any $k\geq3$.
		Hence we will take the liberty of just speaking of `the threshold $r_{col}$' (for $k\geq3$ given).}
Although the precise threshold $r_{col}$ is not known for any $k\geq3$,
the first and the second moment methods can be used to derive upper and lower bounds.
To put our results in perspective, let us briefly recap these techniques.

\smallskip
\noindent{\bf\em The first and the second moment method.}
The first moment method yields an upper bound on $r_{col}$.
More precisely, by Markov's inequality,
	$$\pr\brk{H_k(n,m)\mbox{ is $2$-colorable}}=\pr\brk{Z\geq1}\leq\Erw\brk Z.$$
Hence, if for some density $r$ the first moment $\Erw\brk Z$ satisfies $\Erw\brk Z=o(1)$, then $r_{col}\leq r$
		(for large enough~$n$).
Indeed, it is  easy to compute $\Erw\brk Z$ explicitly, and to verify
that there is a critical density $r_{first}=2^{k-1}\ln(2)-\ln(2)/2+o_k(1)$ such that
$\Erw\brk Z=\exp(\Omega(n))\gg1$ if $r<r_{first}$ while
$\Erw\brk Z=o(1)$ if $r>r_{first}$.
Hence, $r_{col}\leq r_{first}$.

Even though $\Erw\brk Z=\exp(\Omega(n))$ is \emph{exponentially} large in $n$
	for $r<r_{first}$,
this does, of course, not imply that $H_k(n,m)$ is $2$-colorable with high probability:
	it could simply be that a tiny number of hypergraphs drive up the expected number of $2$-colorings
	because they possess excessively many of them.
The purpose of the second moment method is to rule this possibility out.
More precisely, the second moment argument is based on the \emph{Paley-Zygmund inequality}
	$$\pr\brk{H_k(n,m)\mbox{ is $2$-colorable}}=\pr\brk{Z>0}\geq\frac{\Erw\brk{Z}^2}{\Erw\brk{Z^2}}.$$
Hence, if for some density $r<r_{first}$ we can show that
	\begin{equation}\label{eqsmm}
	\Erw\brk{Z^2}\leq C\cdot\Erw\brk{Z}^2
	\end{equation}
with $C=C(k,r)>0$ independent of $n$, then $\pr\brk{H_k(n,m)\mbox{ is $2$-colorable}}\geq1/C$.
That is, the probability of $2$-colorability is bounded away from $0$ as $n\ra\infty$.
Therefore, the sharp threshold theorem implies that $r_{col}\geq r$.
Indeed, Achlioptas and Moore~\cite{nae} proved that there is a critical density
	$r_{second}=2^{k-1}\ln2-(1+\ln2)/2+o_k(1)$ such that~(\ref{eqsmm}) holds for all
	$r<r_{second}$ but is violated for $r>r_{second}$.
In summary, the first/second moment arguments yield the bounds
	\begin{equation}\label{eqnaive}
	r_{second}=2^{k-1}\ln2-\frac{1+\ln2}2+o_k(1)\leq r_{col}\leq r_{first}=2^{k-1}\ln2-\frac{\ln2}2+o_k(1).
	\end{equation}

%\smallskip
\noindent{\bf\em Approaching the condensation threshold.}
How could we improve the lower bound on $r_{col}$?
The second moment analysis in~\cite{nae} is tight, and thus simply performing a better calculation will not suffice.
Indeed, as observed in~\cite{nae}, for $r>r_{second}$ we have $\Erw\brk{Z^2}>\exp(\Omega(n))\cdot\Erw\brk Z^2$,
i.e., the second moment method fails \emph{dramatically}.
But why?
One possibility could be that the \emph{expectation} $\Erw\brk Z$ is driven up by a tiny minority of hypergraphs
with excessively many 2-colorings, i.e., that $Z\leq\exp(-\Omega(n))\Erw\brk Z$ \whp\
In this case~(\ref{eqsmm}) would fail to hold because
	the second moment $\Erw\brk{Z^2}$ would exacerbate the contribution of the few `rich' hypergraphs 
	even more than the first moment.
A second possibility is that $Z$ is `close' to $\Erw\brk Z$ \whp, but without being sufficiently concentrated
for~(\ref{eqsmm}) to hold.
The following theorem, which improves the lower bound  in~(\ref{eqnaive}) by an additive
$(1-\ln(2))/2\approx0.153$, shows that up to $r_{cond}=2^{k-1}\ln2-\ln2>r_{second}$, the second scenario is true.

\begin{theorem}\label{Thm_enhanced}
There is a constant $k_0\geq3$ such that for all $k\geq k_0$ and
	$r<r_{cond}$
the random hypergraph $H_k(n,m)$ is $2$-colorable \whp\  and
	\begin{equation}\label{eqannealed}
	\ln Z\sim\ln\Erw\brk Z\qquad\mbox{ \whp}
	\end{equation}
\end{theorem}
For $r<r_{cond}$ the \emph{expected} number $\Erw\brk Z$ of $2$-colorings is exponentially large in $n$.
Hence, (\ref{eqannealed}) shows that 
for $r<r_{cond}$
\whp\ $Z$ is exponentially large as it coincides with $\Erw\brk Z$ up to sub-exponential terms.

The proof of \Thm~\ref{Thm_enhanced} is based on an enhanced second moment argument
that takes the `geometry' of the set $\cS(H_k(n,m))$ of $2$-colorings of the random hypergraph into account.
As a corollary of this argument, we obtain a result on the `shape' of this set, viewed as a subset of the $n$-dimensional Hamming cube $\cbc{0,1}^n$
equipped with the Hamming distance.
To state this result, let us say that a $2$-coloring $\sigma$ of a hypergraph $H$ on $n$ vertices
is \emph{$(\alpha,\beta,\gamma)$-shattered} for $\alpha,\gamma>0$ and $\beta>\alpha$ if the following is true.
\begin{description}
\item[SH1.] There is no $2$-coloring $\tau\in\cS(H)$ with $\alpha n<\dist(\sigma,\tau)<\beta n$.
\item[SH2.] The set $\cC_\alpha(\sigma)$ of all $2$-colorings $\tau\in\cS(H)$ with $\dist(\sigma,\tau)\leq\alpha n$
			has size $|\cC_\alpha(\sigma)|\leq\exp(-\gamma n)Z(H)$.
\end{description}
Intuitively, this means that $\sigma$ is part of a `cluster' $\cC_\alpha(\sigma)$ of $2$-colorings,
whose size is exponentially smaller than the total number $Z(H)$ of $2$-colorings.
Furthermore, there is a `gap' of size $(\beta-\alpha)n$ between this cluster and the remaining $2$-colorings of $H$.

\begin{corollary}\label{Cor_shattering}
There is a constant $k_0\geq3$ such that for any $k\geq k_0$ there is
$\gamma_k>0$ such that for $r<r_{cond}$
all $2$-colorings of the random hypergraph $H_k(n,m)$ are $(0.01,0.49,\gamma_k)$-shattered \whp
\end{corollary}

\Cor~\ref{Cor_shattering} implies that \whp\ the set of $2$-colorings of $H=H_k(n,m)$
has a decomposition $\cS(H)=\bigcup_{i=1}^NS_i$ into subsets
that each comprise only an exponentially small fraction of all $2$-colorings
and that are mutually at Hamming distance at least $0.48n$.
(Indeed, inductively choose $S_i$ to be the local cluster $\cC_{0.01}(\sigma)$ of some $2$-coloring $\sigma\not\in\bigcup_{j<i}S_j$.)
This decomposition allows us to explain intuitively why the `vanilla' second moment argument
fails for $r_{second}<r<r_{cond}$.
In fact, we can write
	$Z(H)^2=\sum_{i,j=1}^N|S_i|\cdot|S_j|.$
To estimate the expectation of this quantity, we need to bound on the number $N$ of components and their \emph{sizes} $|S_i|$.
As we will see in \Sec~\ref{Sec_smm}, the naive second moment argument overestimates the `cluster sizes' $|S_i|$ grossly.
We overcome this problem by investigating the internal structure of the `clusters' $S_i$.
We expect that this approach extends to other problems such as random $k$-SAT or random graph $k$-coloring,
although the technical details will be far more intricate.

\smallskip
\noindent{\bf\em Into the condensation phase.}
As we will see next, even the enhanced second moment argument from \Thm~\ref{Thm_enhanced} does not
give the precise threshold for 2-colorability.
The intuitive reason is that for densities beyond $r_{cond}$, the \emph{expected} number $\Erw\brk Z$ of $2$-colorings
is indeed driven up excessively by a tiny minority of hypergraphs with an abundance of $2$-colorings.

\begin{theorem}\label{Thm_cond}
There exist a constant $k_0\geq3$ and a sequence 
$\eps_k\ra0$  
such that for any $k\geq k_0$ there are $\delta_k>0,\zeta_k>0$ such that the following two statements are true.
\begin{enumerate}
\item \Whp\ $H_k(n,m)$ is $2$-colorable for all $r<r_{cond}+\eps_k+\delta_k$.
\item For any density $r$ with $r_{cond}+\eps_k<r<r_{col}$ we have
		\begin{equation}\label{eqcond}
		\ln Z<\ln\Erw\brk Z-\zeta_k n\qquad\mbox{ \whp}
		\end{equation}
\end{enumerate}
\end{theorem}
The second statement asserts that for densities between $r_{cond}+\eps_k$ and the actual (unknown) $2$-colorability threshold $r_{col}$,
the \emph{expected} number $\Erw\brk Z$ of $2$-colorings exceeds the \emph{actual} number $Z$ by an
exponential factor $\exp(\zeta_k n)$ \whp\
This contrasts with \Thm~\ref{Thm_enhanced}, which shows that below $r_{cond}$,
$Z$ is of the same exponential order as $\Erw\brk Z$ \whp\
Furthermore, the first part of \Thm~\ref{Thm_cond} ensures that the regime of densities where~(\ref{eqcond}) holds is non-empty,
as the true threshold $r_{col}$ is indeed strictly greater than $r_{cond}+\eps_k$.
This so-called \emph{condensation transition} at density $r_{cond}=2^{k-1}\ln2-\ln2$ was predicted on the basis of non-rigorous
statistical mechanics arguments~\cite{DRZ08,pnas}.

In mathematical physics, the term `phase transition' is usually defined as a point where the function %$F(r)$ is non-analytic.
	$F(r)=\lim_{n\ra\infty}\frac1n\Erw\brk{\ln(1+Z)}$ is non-analytic.
However, it is not currently known if the limit $F(r)$ exists.
	(Bayati, Gamarnik and Tetali~\cite{BGT10} proved that for any density $r$, the corresponding limit of the partition
		function at any fixed positive temperature exists.)
It is not difficult to see that \Thm s~\ref{Thm_enhanced} and~\ref{Thm_cond} imply that around $r=r_{cond}$, the function
$F(r)$ in fact is non-analytic if the limit exists (because for $r<r_{cond}$, $F(r)$ coincides with the \emph{linear} function $\lim_{n\ra\infty}\frac1n\ln\Erw\brk Z$).

The term `condensation' 
is meant to express that \whp\
the set $\cS(H_k(n,m))$ of all $2$-colorings has a drastically different shape than in the `shattered' regime of \Cor~\ref{Cor_shattering}.
To express this, let us call a $2$-coloring of a hypergraph $H$ on $n$ vertices \emph{$(\alpha,\beta,\gamma)$-condensed} if
\begin{description}
\item[CO1.] There is no $2$-coloring $\tau\in\cS(H)$ with $\alpha n<\dist(\sigma,\tau)<\beta n$.
\item[CO2.] The set $\cC_\alpha(\sigma)$ of all $2$-colorings $\tau\in\cS(H)$ with $\dist(\sigma,\tau)\leq\alpha n$
			has size $|\cC_\alpha(\sigma)|\geq\exp(-\gamma n)Z(H)$.
\end{description}
(The difference between {\bf SH1--SH2} and the above is that {\bf CO2} imposes a \emph{lower} bound on $|\cC_\alpha(\sigma)|$.)

\begin{corollary}\label{Cor_cond}
There exist a constant $k_0\geq3$ and a sequence $\eps_k\ra 0$ such that for any $k\geq k_0$ there exist a sequence $r(n)$ of densities
satisfying $|r(n)-r_{cond}|\leq\eps_k$ such that $H_k(n,m)$ 
with $m=r(n)\cdot n$ has the following two properties \whp
\begin{enumerate}
\item $H_k(n,m)$ is $2$-colorable.
\item A random $2$-coloring $\sigma\in\cS(H_k(n,m))$ is $(0.01,0.49,o(1))$-condensed \whp
\end{enumerate}
\end{corollary}

This means that at a particular density $r(n)$, i.e., \emph{right at} the condensation transition,
the size of the local cluster of a `typical' $2$-coloring $\sigma$ of $H_k(n,m)$ satisfies
$\ln|\cC_{0.01}(\sigma)|\sim\ln Z$ \whp\
In other words, the size of the cluster of a `typical' $2$-coloring has the same exponential order as the set of \emph{all} $2$-colorings.
This contrasts with the `shattered' scenario of \Cor~\ref{Cor_shattering}, where \whp\ \emph{all} clusters 
only comprise an exponentially small fraction of the entire set $\cS(H_k(n,m))$.
The statistical physics work~\cite{DRZ08,pnas} suggests that indeed, the conclusions of \Cor~\ref{Cor_cond}
hold in the \emph{entire} regime between the condensation transition
and the $2$-colorability threshold.

\smallskip
\noindent{\bf Discussion.}
The significance of the slightly better lower bound on the
threshold for hypergraph $2$-colorability provided by \Thm~\ref{Thm_enhanced}
is that it allows us to prove the existence of the \emph{condensation transition}.
Beyond the condensation transition, the combinatorial nature of the
problem becomes far more complicated. 
To see why, consider the following random experiment with
$r<r_{col}$ (so that $H_k(n,m)$ is 2-colorable \whp).
\begin{description}
\item[G1.] Choose a random hypergraph $H=H_k(n,m)$,
		conditional on $H$ being $2$-colorable.
\item[G2.] Choose a $2$-coloring $\sigma\in\cS(H)$ uniformly at random and output $(H,\sigma)$.
\end{description}
The above experiment induces a probability distribution $g_{k,n,m}$ on the set $\Lambda_k(n,m)$ of hypergraph/2-coloring pairs
that we call the \emph{Gibbs distribution}.
For $r<r_{col}$ the experiment corresponds to sampling a random $2$-coloring of a random hypergraph,
and thus understanding the above experiment is the key to studying the combinatorial nature of the hypergraph $2$-colorability problem.
But the experiment seems genuinely difficult to analyze.
In fact, even for densities $r=O(2^{k-1}/k)$ \emph{far} below the threshold for $2$-colorability, it is not known how
to efficiently construct, let alone sample, a $2$-coloring of a random hypergraph~\cite{AKKT}.

But there is a related experiment called the \emph{planted model} that is rather easy to implement and to study.
\begin{description}
\item[P1.] Choose $\sigma\in\cbc{0,1}^n$ uniformly at random.
\item[P2.] Choose a hypergraph $H=H_k(n,m,\sigma)$ with $m$ edges uniformly at random
		among all hypergraphs for which $\sigma$ is a proper $2$-coloring,  and output $(H,\sigma)$.
\end{description}
Let $p_{k,n,m}$ denote the distribution on $\Lambda_k(n,m)$ induced by {\bf P1--P2}.
It is not difficult to show that prior to the condensation phase, the distributions induced by the two experiments are `close'.

\begin{proposition}[\cite{Barriers}]\label{Prop_planting}
Suppose that $r<r_{first}$ is such that $\ln Z\sim\ln\Erw\brk Z$ \whp\
Then
	\begin{equation}\label{eqplanting}
	\ln(g_{k,n,m}\brk{\cB|\cbc{\ln Z\sim\ln\Erw\brk Z}})\leq\ln(p_{k,n,m}\brk\cB)+o(n)
		\qquad\mbox{for any event }\cB\neq\emptyset.
	\end{equation}
\end{proposition}
The relationship~(\ref{eqplanting}) allows us to bound the probability of some `bad' event $\cB$
in the Gibbs distribution by estimating its probability in the planted distribution.
Indeed, \Prop~\ref{Prop_planting} was used in~\cite{Barriers} to study various properties of `typical' $2$-colorings of $H_k(n,m)$.
In combination with \Thm~\ref{Thm_enhanced} and the methods of~\cite{Barriers}, \Prop~\ref{Prop_planting} can be used to get a pretty good idea what
a $2$-coloring of the random hypergraph $H_k(n,m)$ ``typically looks like''
	\emph{before} the condensation transition.

But beyond the condensation transition, all bets are off.
As \Thm~\ref{Thm_cond} shows, in the condensed regime we have $\ln Z<\ln\Erw\brk Z-\Omega(n)$ \whp,
	i.e., the assumption of \Prop~\ref{Prop_planting} is violated.
Roughly speaking, the gap $\ln Z<\ln\Erw\brk Z-\Omega(n)$ implies that
	a pair chosen from the planted distribution {\bf P1--P2} corresponds to a pair chosen
	from the Gibbs distribution only with exponentially small probability.
In fact, for densities beyond the condensation transition our proof of \Thm~\ref{Thm_cond} exhibits an event $\cB$
for which~(\ref{eqplanting}) is violated,
i.e., the planted model is no longer a good approximation to the Gibbs distribution.
Furthermore, the statistical mechanics cavity technique suggests that getting a handle
on the Gibbs measure (or other related measures) is far more complicated in the condensation phase.
Overcoming this obstacle appears to be the remaining challenge to obtaining the precise threshold for hypergraph 2-colorability.
The statistical mechanics reasoning \cite{DRZ08,pnas} suggests

\begin{conjecture}
There is $\eps_k\ra0$ such that
$r_{col}\sim2^{k-1}\ln2-(\frac{\ln{2}}{2}+\frac14)+\eps_k$.
\end{conjecture}

One limitation of our approach is that we need to assume that $k\geq k_0$ is sufficiently big
	(whereas the standard second moment argument~\cite{nae} applies to any $k\geq3$).
We need the lower bound on $k$ to carry out a sufficiently accurate analysis of combinatorial structure of the solution
space $\cS(H_k(n,m))$. 
No attempt has been made to compute (let alone optimize) $k_0$ or the various other constants. 

\section{Related work}

The two inequalities in~(\ref{eqnaive}) state the best previous bounds on the threshold for hypergraph $2$-colorability from
	the paper of Achlioptas and Moore~\cite{nae}, which provided the prototype for the second moment analyses in other
	sparse random CSPs (e.g.,~\cite{AchNaor,yuval}).
Since the second moment method is non-constructive, there is the separate algorithmic question: for what densities can a $2$-coloring of a random
hypergraph be constructed in polynomial time \whp?
The best current algorithm is known to succeed up to $r=c\cdot 2^{k-1}/k$ for some constant $c>0$, i.e.,
up to a factor of about $k$ below the $2$-colorability threshold~\cite{AKKT}.

In~\cite{Barriers} the geometry of the set $\cS(H_k(n,m))$ of $2$-colorings of the random hypergraph was investigated (among other things).
It was shown that $\cS(H_k(n,m))$ shatters into exponentially small well-separated `clusters'
for densities $(1+\eps_k)2^{k-1}\ln(k)/k<r<r_{second}$.
\Cor~\ref{Cor_shattering} extends this picture up to $r<r_{cond}$.
In addition, \cite{Barriers} also proved that in the regime $(1+\eps_k)2^{k-1}\ln(k)/k<r<r_{second}$ a typical $2$-coloring $\sigma$
of $H_k(n,m)$ is \emph{rigid} \whp\ in the sense that for most vertices $v$ any $2$-coloring $\tau$ with $\sigma(v)\neq\tau(v)$
has Hamming distance $\Omega(n)$ from $\sigma$.
Our analysis, most notably the study of the structure of a typical `local cluster' in \Sec~\ref{Sec_localCluster}, builds
substantially on the concepts of shattering and rigidity from~\cite{Barriers},
but we will have to 
to elaborate them in considerably more detail to get close quantitative estimates.

In many random CSPs other than random hypergraph $2$-coloring
the best current bounds on the thresholds for the existence of solutions derive from the second moment method as well.
The most prominent examples are random graph $k$-coloring~\cite{AchNaor} and random $k$-SAT~\cite{yuval}.
But the second moment argument extends naturally to a range of `symmetric' random CSPs~\cite{MRT}.
It would be interesting to see if/how our techniques can be generalized
in order to prove the existence of a condensation transition in these other problems, particularly random graph $k$-coloring.
However, since even the standard second moment analysis is quite involved in this case of random graph $k$-coloring,
such a generalization will be technically challenging.

The random $k$-SAT problem is conceptually different because it is not `symmetric'.
More precisely, in random hypergraph $2$-coloring the \emph{inverse} $\vecone-\sigma$ of a $2$-coloring $\sigma$
is a $2$-coloring as well.
This symmetry, which greatly simplifies the second moment argument,
is absent in random $k$-SAT.
As a consequence, as elaborated in~\cite{nae,yuval}, in $k$-SAT the bound $\Erw\brk{Z^2}=O(\Erw\brk{Z}^2)$
does not hold for \emph{any} density.
Roughly speaking, to overcome this problem~\cite{yuval} focuses on a special type of satisfying assignments (``balanced'' ones),
whose number  $Z_*$ satisfies
$\Erw\brk{Z_*^2}=O(\Erw\brk{Z_*}^2)$.
Technically, this is accomplished by weighting satisfying assignments cleverly.
While our techniques can be extended easily to establish the existence of a condensation transition for these
\emph{balanced} satisfying assignments in random $k$-SAT, this does not imply that condensation occurs with respect to the bigger set
of \emph{all} satisfying assignments.
This would require a new approach for the direct analysis of the \emph{total} number
of satisfying assignments in random $k$-SAT.

We emphasize that our techniques are quite different from the `weighted' second moment method in~\cite{yuval}.
Indeed, the `asymmetry' that motivated the weighting scheme in~\cite{yuval} is absent in random hypergraph $2$-coloring.
Instead of weighting, we employ a new idea that exploits
 the combinatorial structure of the `clusters' into which the set $\cS(H_k(n,m))$ of $2$-colorings decomposes.

An example of a random CSP 
in which the precise threshold for the existence of solutions is known is random $k$-XORSAT.
In this problem 
a second moment argument yields the precise thresholds (after `pruning' the underlying hypergraph)~\cite{Dubois,PittelSorkin}.
The explanation for this success is that random $k$-XORSAT does not have  a condensation phase due to the algebraic nature of the problem.
Similarly, in random $k$-SAT with $k>\log_2n$ (i.e., the clause length is \emph{growing} with $n$)
there is no condensation phase and, in effect, the second moment method yields the precise satisfiability threshold~\cite{ACOFriezekSAT,FriezeWormald}.
A further class of problems where the condensed phase is conjectured to be empty are the `locked' problems of \cite{ZdeMez}.

% Added Lenka
In statistical mechanics the condensation transition was first predicted (using non-rigorous techniques) for the random $k$-SAT and the random graph $k$-coloring problems~\cite{pnas}. 
For random hypergraph 2-coloring the statistical mechanics prediction for the condensation threshold was derived in~\cite{DRZ08}. The structure of the condensed phase is described using a non-rigorous framework called \emph{one-step replica symmetry breaking}. Interestingly, it was also conjectured that the structure of the condensed phase for large $k$ is very similar to the structure of the random subcube model \cite{MoraZdeb}.
Our proofs verify this for random hypergraph 2-coloring.

Random CSPs, including random hypergraph $2$-coloring, have been studied in statistical mechanics as models of disordered systems (such as glasses)
under the name `diluted mean field models'.
In this context the condensation transition corresponds to the so-called \emph{Kauzmann transition}~\cite{Kauzmann48}.
The present paper provides the first rigorous proof that this phase transition actually exists in a `diluted mean field model'.

\section{Preliminaries}

We need the following Chernoff bound on the tails of a binomially distributed random variable from~\cite[p.~21]{JLR}.
Let $\varphi(x)=(1+x)\ln(1+x)-x$.
\cite[p.~26]{JLR}

\begin{lemma}\label{Lemma_Chernoff}
Let $X$ be a binomial random variable with mean $\mu>0$.
Then for any $t>0$ we have
	\begin{eqnarray*}
	\pr\brk{X>\Erw\brk X+t}&\leq&\exp(-\mu\cdot\varphi(t/\mu)),\\
	\pr\brk{X<\Erw\brk X-t}&\leq&\exp(-\mu\cdot\varphi(-t/\mu)).
	\end{eqnarray*}
In particular, for any $t>1$ we have
	$\pr\brk{X>t\mu}\leq\exp\brk{-t\mu\ln(t/\eul)}.$
\end{lemma}

The following large deviations principle for the binomial distribution can be found, e.g.,
in \cite[p.~27]{JLR}.

\begin{lemma}\label{Lemma_binlargedev}
Let $X=\Bin(n,p)$ be a binomial random variable
with $\mu=np>0$.
Let $t$ be such that $\mu+t\in\cbc{1,\ldots,n-1}$.
Then
	\begin{eqnarray*}
	\ln\pr\brk{X=\mu+t}&\sim&
		-\mu\varphi(t/\mu)-(n-\mu)\varphi(t/(n-\mu)).
	\end{eqnarray*}
\end{lemma}

The following is a mild generalization of `Laplace lemmas' statements in~\cite{nae,Dubois}.

\begin{lemma}\label{Lemma_Laplace}
Let $\psi\in C^3(0,1)$ be such that $\lim_{x\ra 0}\psi(x)=\lim_{x\ra 1}\psi(x)=0$.
Assume that $z\in\bc{0,1}$ is the unique global maximum of $\psi$, that $\psi(z)>0$, and that $\psi''(z)<0$.
Then
	$$\sum_{d=1}^{n-1}\exp(\psi(d/n))\leq O(\sqrt n)\exp(n\psi(z)).$$
\end{lemma}
\begin{proof}
Since $\psi\in C^3(0,1)$, Taylor's formula shows that
	\begin{equation}\label{eqLaplace1}
	\psi(z+\delta)-\psi(z)=\frac{\delta^2}2\cdot\psi''(z)+O_{\delta}(\delta^2).
	\end{equation}
Moreover, as $\lim_{x\ra 0}\psi(x)=\lim_{x\ra 1}\psi(x)=1$, for any fixed $\delta>0$ we have
	\begin{equation}\label{eqLaplace2}
	\sum_{d=1}^{n-1}\exp(\psi(d/n))\sim\sum_{(z-\delta)n<d<(z+\delta)n}\exp(\psi(d/n)).
	\end{equation}
Suppose that $(z-\delta)n<d<(z+\delta)n$.
Then~(\ref{eqLaplace1}) implies that for small enough $\delta$,
	\begin{eqnarray}\nonumber
	\exp\brk{n\psi(d/n)}&=&\exp\brk{n\psi(z)}\cdot\exp\brk{n\bc{\frac{\psi''(z)}2\bc{\frac dn-z}^2+O((d/n-z)^3)}}\\
		&\leq&\exp\brk{n\psi(z)}\cdot\exp\brk{\frac{\psi''(z)}3\cdot\frac{(d-zn)^2}n}.
			\label{eqLaplace3}
	\end{eqnarray}
Combining~(\ref{eqLaplace2}) and~(\ref{eqLaplace3}) yields the assertion.
\qed\end{proof}

The following lemma is implicit in~\cite{Barriers}.

\begin{lemma}\label{Lemma_selfavg}
For any $\eps>0$ and any $k\geq3$ the following is true.
Suppose that $r<r_{cond}$.
Then \whp\ $H_k(n,m)$ is such that
	$$\ln Z\sim\Erw\ln\brk{1+Z}.$$
\end{lemma}

\section{The enhanced second moment argument
	}\label{Sec_smm}

{\em In the rest of this paper, we assume that $k\geq k_0$ for some large enough constant $k_0$.
	Moreover, to avoid floor and ceiling signs, we assume that $n$ is even.}

\subsection{The local cluster and the demise of the vanilla second moment argument}\label{Sec_vanilla}

We begin by briefly reviewing the `vanilla' second moment method
from~\cite{nae}. 
This will provide the background for the enhanced the second moment argument that yields \Thm~\ref{Thm_enhanced}.
As a first step, we need to work out the \emph{expected} number $\Erw\brk Z$ of $2$-colorings.

\begin{lemma}\label{Lemma_ErwZ}
We have $\Erw\brk Z\sim 2^n(1-2^{1-k})^m$.
\end{lemma}
\begin{proof}
Any fixed $\sigma:V\ra\cbc{0,1}$ is a $2$-coloring of $H_k(n,m)$
iff $H_k(n,m)$ does not feature an edge that consists of vertices in one color class  $\sigma^{-1}(i)$ only ($i=1,2$).
In other words, $\sigma$ `forbids' $\bink{|\sigma^{-1}(0)|}{k}+\bink{|\sigma^{-1}(1)|}{k}$ out of the $\bink nk$ possible edges.
Clearly,  the number of `forbidden' edges is minimized
if both color classes $\sigma^{-1}(0),\sigma^{-1}(1)$ are the same size $n/2$.
Furthermore, for all but a $o(1)$-fraction of all $2^n$ possible $\sigma:V\ra\cbc{0,1}$ it is indeed true that 
both color classes have size $(1\pm o(1))n/2$.
By the linearity of expectation,
	$$\Erw\brk Z\sim 2^n\bink{\bink{n}k-2\bink{n/2}{k}}m/\bink{\bink nk}{m}\sim2^n(1-2^{1-k})^m,$$
where the last step follows from Stirling's formula.
\qed\end{proof}

Our goal is to identify the regime of densities $r$ where $\Erw\brk{Z^2}=O(\Erw\brk{Z}^2)$, i.e., where the second moment method `works'.
A technical issue is that $Z$ includes $2$-colorings $\sigma$ whose color classes have (very) different sizes.
To simplify our calculations we are going to confine ourselves to colorings $\sigma$ whose color classes $\sigma^{-1}(0),\sigma^{-1}(1)$ have the same size.
More precisely, let us call $\sigma:V\ra\cbc{0,1}$ \emph{equitable} if $|\sigma(0)|=|\sigma(1)|=n/2$, and let
$Z_e$ be the number of equitable $2$-colorings of $H_k(n,m)$.
Using Stirling's formula and, once more, the linearity of the expectation,
it is not difficult to compute $\Erw\brk{Z_e}$: we have
	\begin{equation}\label{eqErwZe}
	\Erw\brk{Z_e}
		\sim\sqrt{\frac{2}{\pi}}\cdot\frac{2^n}{\sqrt n}(1-2^{1-k})^m=\Theta(1/\sqrt n)\cdot\Erw\brk Z.
	\end{equation}
Now, for what $r$ do we have $\Erw\brk{Z_e^2}=O(\Erw\brk{Z_e}^2)$?
We use the following elementary relation.

\begin{fact}\label{Fact_condexp}
For any equitable $\sigma:V\ra\cbc{0,1}$ we have
$\Erw\brk{Z_e^2}=\Erw\brk{Z_e}\cdot\Erw\brk{Z_e|\mbox{$\sigma$ is a $2$-coloring}}.$
\end{fact}
\begin{proof}
As $\Erw\brk{Z_e^2}$ equals the expected number of \emph{pairs} of equitable $2$-colorings, we find
	\begin{eqnarray*}
	\Erw\brk{Z_e^2}
		&=&\sum_{\sigma,\tau}\pr\brk{\mbox{$\sigma$ is a $2$-coloring}}\cdot
				\pr\brk{\mbox{$\tau$ is a $2$-coloring}|\mbox{$\sigma$ is a valid $2$-coloring}}\\
		&=&\sum_{\sigma}\pr\brk{\mbox{$\sigma$ is a $2$-coloring}}\cdot\Erw\brk{Z_e|\mbox{$\sigma$ is a $2$-coloring}}.
	\end{eqnarray*}
By symmetry, $\Erw\brk{Z_e|\mbox{$\sigma$ is a $2$-coloring}}$ is the same for \emph{all} equitable $\sigma$.
Moreover, by the linearity of the expectation we have $\Erw\brk{Z_e}=\sum_{\sigma}\pr\brk{\mbox{$\sigma$ is a $2$-coloring}}$.
\qed\end{proof}

Thus, 
we need to compute  $\Erw\brk{Z_e|\mbox{$\sigma$ is a $2$-coloring}}$.
In other words, for a \emph{fixed} equitable $\sigma\in\cbc{0,1}^n$ we need to
study the random hypergraph $H_k(n,m)$ \emph{given that} $\sigma$ is a $2$-coloring.
This conditional distribution
	can be expressed easily: 
	just choose a set of $m$ edges uniformly at random from all edges that are bichromatic under $\sigma$
	(cf.\ step {\bf P2} of the `planted model' above).
Let $H_k(n,m,\sigma)$ denote the resulting random hypergraph.
Furthermore, given $\sigma$, let $Z_e(d)$ be the number of equitable $2$-colorings $\tau$ with Hamming distance $\dist(\sigma,\tau)=d$.
Similarly,  let $Z(d)$ be the \emph{total} number of $2$-colorings $\tau$ with $\dist(\sigma,\tau)=d$.
Then 
	\begin{equation}\label{eqcond2dist}
	\Erw\brk{Z_e|\mbox{$\sigma$ is a $2$-coloring}}
		=\sum_{d=0}^n\Erw_{H_k(n,m,\sigma)}\brk{Z_e(d)}\leq\sum_{d=0}^n\Erw_{H_k(n,m,\sigma)}\brk{Z(d)}.
	\end{equation}

\begin{fact}[{\cite{nae}}]\label{Fact_Zd}
For any $0<d<n$ we have
	\begin{eqnarray*}
	\Erw_{H_k(n,m,\sigma)}\brk{Z(d)}&=&\Theta(\sqrt{n/(d\cdot(n-d)}))\cdot\exp(\psi(d/n)),\qquad\\
	\Erw_{H_k(n,m,\sigma)}\brk{Z_e(d)}&=&\Theta(n/(d\cdot(n-d)))\cdot\exp(\psi(d/n)),\qquad\mbox{with}\\
	\psi=\psi_{k,r}:(0,1)\ra\RR,&& x\mapsto-x\ln(x)-(1-x)\ln(1-x)+r\cdot\ln\brk{1-\frac{1-x^k-(1-x)^k}{2^{k-1}-1}}.
	\end{eqnarray*}
\end{fact}

Fact~\ref{Fact_Zd} and~(\ref{eqcond2dist}) reduce the problem of computing 
$\Erw\brk{Z_e|\mbox{$\sigma$ is a $2$-coloring}}$ (and thus $\Erw\brk{Z_e^2})$) to an exercise in calculus:
we just need to study the function $\psi$.

\begin{lemma}[\cite{nae}]\label{Lemma_smm}%
Suppose $r<r_{first}$.
The function $\psi$ satisfies $\psi(1/2)\sim\frac1n\ln\Erw\brk{Z}$, $\psi(1-x)=\psi(x)$, 
$\psi'(1/2)=0$, and $\psi''(1/2)<0$.
Moreover,
\begin{enumerate}
\item if $\psi(1/2)>\psi(x)$ for all $x\in(0,1)$,
		then
		$\Erw\brk{Z_e|\mbox{$\sigma$ is a $2$-coloring}}\leq O(\Erw\brk{Z_e})$.
\item if there is some $x\in(0,1)$ with $\psi(x)>\psi(1/2)$, then
		$\Erw\brk{Z_e|\mbox{$\sigma$ is a $2$-coloring}}>\Erw\brk{Z}\cdot\exp(\Omega(n))$.
\end{enumerate}
\end{lemma}

\Lem~\ref{Lemma_smm} shows that the second moment method `works' if and only if $r$ is such that the function $\psi$
takes its global maximum at $\frac12$.
Thus, let $r_{second}$ be the supremum of all $r>0$ with this property. 
Using basic calculus, one verifies that $r_{second}=2^{k-1}\ln2-\frac12(1+\ln2)+o_k(1)$ (see~\cite[\Sec~7]{nae}),
and that for $r>r_{second}$ the function $\psi$ attains its maximum, strictly greater than $\psi(1/2)$, in the interval $(0,2^{-k/2})$.
In effect, the second part of \Lem~\ref{Lemma_smm} shows that $\Erw\brk{Z^2}\geq\Erw\brk{Z_e^2}\geq\exp(\Omega(n))\Erw\brk{Z}$
for $r>r_{second}$, i.e., the `vanilla' second moment argument breaks beyond $r_{second}$.

\subsection{Improving the second moment argument: proof of \Thm~\ref{Thm_enhanced}}\label{Sec_enhanced}

To improve over the naive second moment argument, we take another look at the function $\psi$.
Let $\alpha=2^{-k/2}$.
Once more using basic calculus (see \Sec~\ref{Sec_psi}), we find

\begin{lemma}\label{Lemma_psi}
Suppose that $r_{second}<r<r_{first}$.
\begin{enumerate}
\item We have $\sup_{0<x<\alpha}\psi(x)>\psi(1/2)>0$.
\item For all $x\in(\alpha,1/2-\alpha)\cup(1/2+\alpha,1-\alpha)$ we have $\psi(x)<-\psi(1/2)<0$.
\item In the interval $[\alpha,1-\alpha]$ the function $\psi$ attains its unique maximum at $1/2$.
\end{enumerate}
\end{lemma}

\Lem~\ref{Lemma_psi} allows us to deduce important information on the geometry of the set
$\cS(H_k(n,m,\sigma))$ of $2$-colorings
	(similar arguments as the following have been used in~\cite{Barriers} to prove that the set of all $2$-colorings of $H_k(n,m)$
		shatters into exponentially many well-separated pieces
		for a certain $r$).
Indeed, combining Fact~\ref{Fact_Zd} and \Lem~\ref{Lemma_psi}, we see that
for distances $\alpha n<d\leq(\frac12-\alpha)n$, the \emph{expected} number of $2$-colorings at distance $d$
from $\sigma$ is exponentially small:
	$$\Erw_{H_k(n,m,\sigma)}\brk{Z(d)}=\exp((1+o(1))\psi(d/n)n)\leq\exp(-\Omega(n)).$$
Hence, $H_k(n,m,\sigma)$ does not have \emph{any} $2$-coloring $\tau$
such that $\dist(\sigma,\tau)\in(\alpha n,(\frac12-\alpha)n)$ \whp\ 
Similarly, \whp\ there is no $2$-coloring $\tau$ with $\dist(\sigma,\tau)\in((\frac12+\alpha) n,(1-\alpha)n)$.
Thus, \whp\ the set of $2$-colorings of $H_k(n,m,\sigma)$ decomposes into the `local cluster'
	$$\cC(\sigma)=\cbc{\tau\in \cS(H_k(n,m,\sigma)):\dist(\sigma,\tau)\leq\alpha n}$$
of colorings `close' to $\sigma$, 
the corresponding inverse colorings $\cbc{\vecone-\tau:\tau\in\cC(\sigma)}$,
and the remaining colorings $\tau$ with $\frac12-\alpha\leq\dist(\sigma,\tau)/n\leq\frac12+\alpha$.

With this picture in mind, we can interpret the maximum of $\psi$ in $(0,\alpha)$ as the \emph{expected}
size of the local cluster.
More precisely, by Fact~\ref{Fact_Zd},
	\begin{equation}\label{eqErwlocal}
	\Erw_{H_k(n,m,\sigma)}|\cC(\sigma)|=\sum_{0\leq d\leq\alpha n}\Erw_{H_k(n,m,\sigma)}\brk{Z_d(\sigma)}
			=\exp\brk{(1+o(1))n\cdot\sup_{0<x<\alpha}\psi(x)}.
	\end{equation}
Hence, the `vanilla' second moment argument breaks down for $r> r_{second}$ because
the \emph{expected} size of the local cluster in $H_k(n,m,\sigma)$ exceeds the expected number $\Erw\brk Z$ of $2$-colorings in $H_k(n,m)$.

Our improvement over the plain second moment argument rests on the observation
that for densities $r>r_{second}$ the \emph{expected} size $\Erw\abs{\cC(\sigma)}$ exaggerates
the \emph{typical} size of the local cluster.
More precisely, in \Sec~\ref{Sec_localCluster} below we will investigate the combinatorial structure
of the `planted' formula $H_k(n,m,\sigma)$ closely to prove the following key fact.

\begin{proposition}\label{Prop_local}
Let $\sigma\in\cbc{0,1}^V$ be equitable.
If $r<r_{cond}$, then \whp\ in the random formula $H_k(n,m,\sigma)$ the set
	$\cC(\sigma)=\cbc{\tau\in\cS(H_k(n,m,\sigma)):\dist(\sigma,\tau)\leq 2^{-k/2}n}$
has size $|\cC(\sigma)|\leq\Erw\brk{Z_e}$.
\end{proposition}

Fix a density $r<r_{cond}$.
Let us call a $2$-coloring $\sigma$ of a hypergraph $H$ \emph{good} if $\sigma$ is equitable
and the its local cluster $\cC(\sigma)=\cbc{\tau\in\cS(H):\dist(\sigma,\tau)\leq 2^{-k/2}n}$ has size $|\cC(\sigma)|\leq\Erw\brk{Z_e}$.
Furthermore, let $Z_g$ be the number of good $2$-colorings of $H_k(n,m)$.

\begin{corollary}\label{Cor_expgood}
For any $r<r_{cond}$ we have $\Erw\brk{Z_g}\sim\Erw\brk{Z_e}=\Theta(n^{-1/2})\Erw\brk Z$.
\end{corollary}
\begin{proof}
Let
$\cH$ be the set of all $k$-uniform hypergraphs on $V=\cbc{1,\ldots,n}$ with precisely $m$ edges.
Let $\Lambda_e$ be the set of all pairs $(H,\sigma)$ with $H\in\cH$ and $\sigma\in\cS(H)$ equitable.
Furthermore, let $\Lambda_g$ be the set of all pairs $(H,\sigma)$ with $H\in\cH$ and $\sigma$ a good $2$-coloring of $H$.
Then $\Erw\brk{Z_e}=\Lambda_e/\abs\cH$ and $\Erw\brk{Z_g}=\Lambda_g/\abs\cH$.
Hence, it suffices to show that $\abs{\Lambda_e}\sim\abs{\Lambda_g}$.
But this is evident from \Prop~\ref{Prop_local}.
Indeed, \Prop~\ref{Prop_local} implies that
	$\abs{\cbc{H\in\cH:(H,\sigma)\in\Lambda_g}}\sim\abs{\cbc{H\in\cH:(H,\sigma)\in\Lambda_e}}$
	  for any equitable $\sigma$.
\qed\end{proof}

\begin{corollary}\label{Cor_expgood2}
Suppose that $r<r_{cond}$.
For any equitable $\sigma$ we have
	$$\pr\brk{\sigma\mbox{ is a good $2$-coloring of $H_k(n,m)$}}\sim
		\pr\brk{\sigma\mbox{ is a valid $2$-coloring of $H_k(n,m)$}}.$$
\end{corollary}
\begin{proof}
Since the total number of equitable $\tau\in\cbc{0,1}^V$ equals $2\bink{n}{n/2}$, and because
the uniform distribution over hypergraphs is invariant under permutations of the vertices, we have
	\begin{eqnarray*}
	\Erw\brk{Z_g}&=&2\bink{n}{n/2}\pr\brk{\sigma\mbox{ is a good $2$-coloring of $H_k(n,m)$}},\\
	\Erw\brk{Z_e}&=&2\bink{n}{n/2}\pr\brk{\sigma\mbox{ is a $2$-coloring of $H_k(n,m)$}}.
	\end{eqnarray*}
Hence, the assertion follows from \Cor~\ref{Cor_expgood}.
\qed\end{proof}

We are going to compute the second moment $\Erw\brk{Z_g^2}$.
The exact same calculation that we used to prove Fact~\ref{Fact_condexp} shows that
$\Erw\brk{Z_g^2}\leq C\cdot\Erw\brk{Z_g}^2$ if for any equitable $\sigma$ we have
	\begin{equation}\label{eqdistgood}
	\Erw\brk{Z_g|\mbox{$\sigma$ is a good $2$-coloring}}\leq C\cdot\Erw\brk{Z_g}.
	\end{equation}
Thus, we are left to verify that for $r<r_{cond}$ there is $C=C(k,r)$ such that~(\ref{eqdistgood}) holds.

Let $\alpha=2^{-k/2}$.
Letting $Z_g(d)$ denote the number of good $2$-colorings at Hamming distance $d$ from $\sigma$, we obtain
	\begin{eqnarray}\label{eqgoodlocal}
	\Erw\brk{\sum_{0\leq d\leq \alpha n}Z_g(d)|\mbox{$\sigma$ is good}}
		&\leq&\Erw\brk{\abs{\cC(\sigma)}|\mbox{$\sigma$ is good}}\leq\Erw\brk{Z_e};
	\end{eqnarray}
the last inequality follows because if $\sigma$ is good, then $\cC(\sigma)\leq\Erw\brk{Z_e}$ \emph{with certainty}.
Further, by \Cor~\ref{Cor_expgood2},
	\begin{eqnarray*}\nonumber
	\sum_{\alpha n<d\leq n/2}\Erw\brk{Z_g(d)|\mbox{$\sigma$ is good}}
		&\leq&\sum_{\alpha n<d\leq n/2}\Erw\brk{Z_e(d)|\mbox{$\sigma$ is good}}\\
			&\leq&\sum_{\alpha n<d\leq n/2}\Erw\brk{Z_e(d)|\mbox{$\sigma$ is a valid $2$-coloring}}\cdot\frac{\pr\brk{\mbox{$\sigma$ is a valid $2$-coloring}}}
					{\pr\brk{\mbox{$\sigma$ is good}}}\nonumber\\
			&\sim&\sum_{\alpha n<d\leq n/2}\Erw_{H_k(n,m,\sigma)}\brk{Z_e(d)}\nonumber\\
			&=&\Theta(1/n)\sum_{\alpha n<d\leq n/2}\exp\brk{\psi(d/n)}\qquad\qquad\mbox{[by Fact~\ref{Fact_Zd}]}.
	\end{eqnarray*}
Furthermore, %
\Lem~\ref{Lemma_Laplace},
\Lem~\ref{Lemma_psi} and \Cor~\ref{Cor_expgood} imply that 
	\begin{equation}\label{eqgoodnonlocal}
	\sum_{\alpha n<d\leq n/2}\Erw\brk{Z_g(d)|\mbox{$\sigma$ is good}}
		\leq(1+o(1))\sum_{\alpha n<d\leq n/2}\Erw\brk{Z_e(d)|\mbox{$\sigma$ is a valid $2$-coloring}}
		\leq C'\cdot\Erw\brk{Z_e}
	\end{equation}
for a certain constant $C'=C'(k,r)$.
Since furthermore $Z_g(d)=Z_g(n-d)$ due to the symmetry of the $2$-coloring problem with respect to swapping the color classes,
(\ref{eqgoodlocal}) and~(\ref{eqgoodnonlocal})
yield (\ref{eqdistgood}) with $C=2(C'+1)$.

Hence, we have shown that $\Erw\brk{Z_g^2}\leq C\cdot\Erw\brk{Z_g}^2$ for all $r<r_{cond}$.
\Cor~\ref{Cor_expgood} and the Paley-Zygmund inequality therefore imply that
	\begin{equation}\label{eqZZg}
	\pr\brk{Z>0}\geq
		\pr\brk{Z>\Erw\brk Z/3}\geq
		\pr\brk{Z_g>\Erw\brk{Z_g}/2}\geq1/(4C).
	\end{equation}
In particular, the threshold $r_{col}$ for $2$-colorability cannot be smaller than $r_{cond}$,
whence indeed $H_k(n,m)$ is $2$-colorable \whp\ for any $r<r_{cond}$.
The second claim~(\ref{eqannealed})
follows from~(\ref{eqZZg}) together with \Lem~\ref{Lemma_selfavg}.

\subsection{Beyond the condensation transition: proof of \Thm~\ref{Thm_cond}}\label{Sec_beyond}

The goal in this section is establish \Thm~\ref{Thm_cond}, i.e., to prove that there is a non-empty regime of densities
$r_{cond}<r<r_{col}$ in which $H_k(n,m)$ is $2$-colorable
but $\ln Z<\ln\Erw\brk Z-\Omega(n)$ \whp\
To get an intuition why this should be the case, consider a density $r>r_{cond}+\eps_k$.
In \Prop~\ref{Prop_critical} below we will see that \whp\  (for suitable $\eps_k$), the size $\abs{\cC(\sigma)}$ of the `local cluster'
in the planted model $H_k(n,m,\sigma)$ is bigger by an exponential factor $\exp(\Omega(n))$ than the
expected number $\Erw\brk Z$ of $2$-colorings of $H_k(n,m)$.
However, if it was true that $\ln Z\sim\ln\Erw Z$, then \Prop~\ref{Prop_planting} would imply that
the planted model and the Gibbs distribution (first choose $H_k(n,m)$ and then choose $\sigma\in\cS(H_k(n,m))$ randomly) are `close'.
In particular, in a random pair $(H,\sigma)$ chosen from the Gibbs distribution the local cluster $\cC(\sigma)$
should have size $\geq\Erw\brk Z\exp(\Omega(n))$.
This would lead to the absurd conclusion that under the Gibbs distribution
	$Z\geq\abs{\cC(\sigma)}\geq\Erw\brk Z\exp(\Omega(n))$ \whp\ (in obvious contradiction to Markov's inequality).
Hence, intuitively the condensation transition occurs because the size of the local cluster in the planted model $H_k(n,m,\sigma)$
surpasses the expected number $\Erw\brk Z$ of $2$-colorings of $H_k(n,m)$.
Indeed, it is not difficult to turn this intuition into a proof of part~2 of \Thm~\ref{Thm_cond}
 	(see the proof of \Thm~\ref{Thm_cond} below).
But the above still allows for the possibility that the condensation phase may just be empty,
	i.e., that the typical size of the local cluster in the planted model $H_k(n,m,\sigma)$
	is bounded by $\Erw\brk Z$ for the \emph{entire} regime of $r$ where $H_k(n,m)$ is $2$-colorable \whp\
The purpose of this section is to show that that is not so.

To prove this, we are going to show that \whp\ $H_k(n,m)$ has a $2$-coloring $\sigma$ whose
local cluster $\cC(\sigma)$ is smaller than $\Erw\brk Z$, i.e., \emph{much} smaller than the local cluster in the planted model.
As we will see in \Sec~\ref{Sec_localCluster} below, the size of the local cluster of a $2$-coloring $\sigma$ is governed
by the edges that contain precisely one vertex $v$ with color $i$ and $k-1$ vertices with color $1-i$ (with either $i=0$ or $i=1$).
Let us call such edges \emph{critical} under $\sigma$.
Intuitively, critical edges `freeze' $v$ by preventing $v$ from switching to the opposite color $1-i$, thereby
reducing the entropy of the local cluster.

Given this intuition, it seems natural to assume that $2$-colorings that have a particularly high number
of critical edges should have rather small local clusters.
Thus, we say that a $2$-coloring $\sigma$ of $H_k(n,m)$ is \emph{$(1+\beta)$-critical}
if $\sigma$ is equitable and the total number of critical edges equals
	$(1+\beta)km/(2^{k-1}-1)$.
Let $Z_{1+\beta}$ be the number of $(1+\beta)$-critical $2$-colorings.
Furthermore, let us call a $(1+\beta)$-critical $2$-coloring $\sigma$ \emph{good} if indeed
the local cluster
	$$\cC(\sigma)=\cbc{\tau\in\cS(H):\dist(\sigma,\tau)\leq2^{-k/2}}$$
satisfies
	$\cC(\sigma)\leq\Erw\brk{Z_{1+\beta_k}}$,
and let $Z_{g,1+\beta}$ be the number of good  $(1+\beta)$-critical $2$-colorings.

\begin{proposition}\label{Prop_critical}
For any $k\geq k_0$
there exist a density $r_{crit}>r_{cond}$, $\delta_k>0$, and $\beta_k>0$
such that for $r=r_{cond}$ the following three statements hold.
\begin{enumerate}
\item We have $\Erw\brk{Z_{g,1+\beta_k}}\sim\Erw\brk{Z_{1+\beta_k}}=\exp(\Omega(n))$.
\item Let $\sigma\in\cbc{0,1}^V$ be equitable and let $H=H_k(n,m,\sigma)$ be a hypergraph chosen
		from the planted model.
		Then \whp\ the local cluster
			$\cC(\sigma)=\cbc{\tau\in\cS(H):\dist(\sigma,\tau)\leq2^{-k/2}}$
		has size $\abs{\cC(\sigma)}>\Erw\brk{Z}\cdot\exp(\delta_kn)$.
\end{enumerate}
\end{proposition}
We defer the proof of \Prop~\ref{Prop_critical} to \Sec~\ref{Sec_localCluster}.

In the sequel, we fix $k\geq k_0$ big enough and let $r=r_{crit}$ and $\beta=\beta_k$ be as in \Prop~\ref{Prop_critical}.
In the rest of this section, we are going to carry out a second moment argument for $Z_{g,1+\beta}$ to show the following.

\begin{proposition}\label{Prop_critsecond}
With $r,\beta$ as above, we have
	$$\Erw\brk{Z_{g,1+\beta}^2}\leq C\cdot\Erw\brk{Z_{g,1+\beta}}^2$$
for some constant $C=C(k)>0$.
\end{proposition}

As before, this amounts to showing that
	\begin{eqnarray}\label{eqgooddist}
	\Erw\brk{Z_{g,1+\beta}|\sigma\mbox{ is a good $(1+\beta)$-critical $2$-coloring}}&\leq& C\cdot\Erw\brk{Z_{g,1+\beta}},
	\end{eqnarray}
for some number $C=C(k,r)>0$.
To establish~(\ref{eqgooddist}), we let $Z_{g,1+\beta}(d)$ signify the number of good $(1+\beta)$-critical $2$-colorings
at Hamming distance $d$ from $\sigma$.
Let $\gamma=2^{-k/2}$.
The very definition of `good' ensures that
	\begin{eqnarray}\label{eqgoodlocalbeta}
	\sum_{d\leq\gamma n}\Erw\brk{Z_{g,1+\beta}(d)|\sigma\mbox{ is a good $(1+\beta)$-critical $2$-coloring}}
		\leq\Erw\brk{Z_{g,1+\beta}}.
	\end{eqnarray}
The following bound covers `intermediate' distances.

\begin{lemma}\label{Lemma_gintermediate}
We have
	$$\sum_{\gamma n<d<(1/2-\gamma)n}
		\Erw\brk{Z_{g,1+\beta}(d)|\sigma\mbox{ is a good $(1+\beta)$-critical $2$-coloring}}=o(1).$$
\end{lemma}
\begin{proof}
Let $\gamma n<d<(1/2-\gamma)n$.
Let $\tau$ be equitable and at distance $d$ from $\sigma$.
Let us briefly say $\tau$ is \emph{valid} if $\tau$ is a $2$-coloring of $H_k(n,m)$.
Then
	\begin{eqnarray*}
	&&\hspace{-3cm}\pr\brk{\tau\mbox{ is valid}|\sigma\mbox{ is a good $(1+\beta)$-critical $2$-coloring}}\\
	&=&\frac{\pr\brk{\tau\mbox{ is valid},\,\sigma\mbox{ is a good $(1+\beta)$-critical $2$-coloring}}}{\pr\brk{\sigma\mbox{ is a good $(1+\beta)$-critical $2$-coloring}}}\\
	&\leq&
		\frac{\pr\brk{\sigma,\tau\mbox{ are valid}}}{\pr\brk{\sigma\mbox{ is a good $(1+\beta)$-critical $2$-coloring}}}\\
	&=&\pr\brk{\tau\mbox{ is valid}|\sigma\mbox{ is valid}}\cdot\frac{\pr\brk{\sigma\mbox{ is valid}}}{\pr\brk{\sigma\mbox{ is a good $(1+\beta)$-critical $2$-coloring}}}\\
	&=&\pr\brk{\tau\mbox{ is valid}|\sigma\mbox{ is valid}}\cdot\frac{\Erw\brk{Z}}{\Erw\brk{Z_{g,1+\beta}}}
		\leq\pr\brk{\tau\mbox{ is valid}|\sigma\mbox{ is valid}}\cdot\Erw\brk{Z},
	\end{eqnarray*}
because $\Erw\brk{Z_{g,1+\beta}}>1$ by our choice of $\beta$.
Hence, \Lem~\ref{Lemma_psi} yields
	\begin{eqnarray*}
	&&\hspace{-3cm}\ln\Erw\brk{Z_{g,1+\beta}(d)|\sigma\mbox{ is a good $(1+\beta)$-critical $2$-coloring}}\\
	&\leq&\ln\Erw\brk Z+\ln\Erw\brk{Z_e(d)}\\
	&\leq&-\psi(1/2)+\psi(d/n)+o(1)<0,
	\end{eqnarray*}
as $\gamma<d/n<1/2-\gamma$.
Summing over $d$ yields the assertion.
\qed\end{proof}

Thus, we are left to estimate the contribution of distances $(1/2-\gamma)n\leq d\leq n/2$.
We need to characterize the conditional distribution of $H_k(n,m)$ given that
some equitable $\sigma$ is a $(1+\beta)$-critical $2$-coloring.
But since $H_k(n,m)$ is just a uniformly random hypergraph with $m$ edges, this is straightforward:
	let $H_k(n,m_1,m_2,\sigma)$ denote the random hypergraph generated as follows:
\begin{itemize}
\item Choose a set $E_1$ of $m_1$ edges that are critical with respect to $\sigma$ uniformly at random.
\item Choose a set $E_2$ of $m_2$ edges that are bichromatic under $\sigma$ but not critical uniformly at random.
\item Let $H_k(n,m_1,m_2,\sigma)=(V,E_1\cup E_2)$.
\end{itemize}
Then for $m_1=(1+\beta)km/(2^{k-1}-1)$ and for $m_2=m-m_1$, the 
conditional distribution of $H_k(n,m)$ given that $\sigma$ is $(1+\beta)$-critical is precisely
$H_k(n,m_1,m_2,\sigma)$.

To estimate
	$\Erw_{H_k(n,m_1,m_2,\sigma)}\brk{Z_{g,1+\beta}(d)}$,
we need to study the conditional probability that a certain equitable $\tau$ at distance $d$ from $\sigma$
is $(1+\beta)$-critical.

\begin{lemma}\label{Lemma_cE}
Let $\tau$ be equitable at distance $d=\alpha n$ from $\sigma$.
Then
	$\pr_{H_k(n,m_1,m_2,\sigma)}\brk{\tau\mbox{ is $1+\beta$-critical}}\sim\cE(\alpha),$
where {\small
	\begin{eqnarray*}
	\cE(\alpha)&=&(1-v_1)^{m_1}(1-v_2)^{m_2}\pr\brk{\Bin(m_1,u_1/(1-v_1))+\Bin(m_2,u_2/(1-v_2))=m_1}\mbox{ with}\\
	u_1&=&(1-\alpha)^k+\alpha^k+(k-1)\alpha^2(1-\alpha)^{k-2}+(k-1)\alpha^{k-2}(1-\alpha)^2,\\
	v_1&=&\alpha(1-\alpha)^{k-1}+(1-\alpha)\alpha^{k-1},\\
	u_2&=&\frac{k\bc{1-\alpha^k-(1-\alpha)^k-\alpha^{k-1}(1-\alpha)-\alpha(1-\alpha)^{k-1}-(k-1)\alpha^{k-2}(1-\alpha)^2-(k-1)\alpha^2(1-\alpha)^{k-2}}}
			{2^{k-1}-k-1},\\
	v_2&=&\frac{1-2\brk{\alpha^k+(1-\alpha)^k+2k\alpha(1-\alpha)^{k-1}+2k\alpha^{k-1}(1-\alpha)}}{2^k-2k-2}.
	\end{eqnarray*}}
\end{lemma}
\begin{proof}
By enumerating all possibilities, we see that
the probability that a given edge that is critical under $\sigma$ also is critical under $\tau$ equals $u_1$
	(either all its vertices have the same color under both $\sigma$ and $\tau$, or they have opposite colors
	under $\tau$, or the colors of the supporting vertex and exactly one other vertex differ,
	or the supporting vertex has the same color and the colors of exactly $k-2$ others differ).
Similarly, the probability that an edge that is critical under $\sigma$ is monochromatic under
$\tau$ works out to be $v_1$.

Now, take a random edge that has $2\leq l\leq k-2$ vertices of color $1$ under $\sigma$.
The probability the edge is monochromatic under $\tau$ equals
	$$\alpha^l(1-\alpha)^{k-l}+(1-\alpha)^l\alpha^{k-l}.$$
Convoluting this formula with the distribution of the number of edges with a given number of vertices of color one under $\sigma$,
we obtain
	\begin{eqnarray*}
	v_2&=&\sum_{l=2}^{k-2}\frac{\bink kl(\alpha^l(1-\alpha)^{k-l}+(1-\alpha)^l\alpha^{k-l})}{2^k-2k-2}.
	\end{eqnarray*}
This is the probability that a random edge that is neither critical nor monochromatic under $\sigma$ is monochromatic under $\tau$.

Furthermore, the probability that 
a random edge that has precisely $l$ vertices of color $1$ is critical under $\tau$ equals
	$$l\alpha^{l-1}(1-\alpha)^{k-l+1}+l(1-\alpha)^{l-1}\alpha^{k-l+1}+(k-l)(1-\alpha)^{l+1}\alpha^{k-l-1}+(k-l)\alpha^{l+1}(1-\alpha)^{k-l-1}.$$
Convoluting this formula with the distribution of the number of edges with a given number of vertices of color one under $\sigma$,
we get
	{\small\begin{eqnarray*}
	u_2&=&\sum_{l=2}^{k-2}\bink kl
		\frac{l\alpha^{l-1}(1-\alpha)^{k-l+1}+l(1-\alpha)^{l-1}\alpha^{k-l+1}+(k-l)(1-\alpha)^{l+1}\alpha^{k-l-1}+(k-l)\alpha^{l+1}(1-\alpha)^{k-l-1}}{2^k-2k-2}.
	\end{eqnarray*}}
This is the probability that a 
random edge that is neither critical nor monochromatic under $\sigma$ is critical under $\tau$.
The \emph{conditional} probability of a random edge being bichromatic and critical resp.\ not critical under $\sigma$ is thus
	$$\frac{u_1}{1-v_1}\mbox{ resp.\ }\frac{u_2}{1-v_2}.$$
Since the $m$ edges are drawn independently up to the trivial dependence that no edge
is drawn twice, we thus see that the probability that $\tau\mbox{ is $1+\beta$-critical}$ is $(1+o(1))\cE(\alpha)$.
\qed\end{proof}

\begin{corollary}\label{Cor_cE}
For any $0<d<n$ we have
	\begin{eqnarray*}
	\frac1n\ln\Erw_{H_k(n,m_1,m_2,\sigma)}\brk{Z_{g,1+\beta}(d)}&\sim& g(d/n),\qquad\qquad\mbox{with}\\
	g(\alpha)&=&h(\alpha)+\frac1n\ln\cE(\alpha),\mbox{ where }h(x)=-x\ln(x)-(1-x)\ln(1-x).
	\end{eqnarray*}
\end{corollary}
\begin{proof}
This simply follows from \Lem~\ref{Lemma_cE} and the fact that the number of $\tau$ at distance
$d$ from $\sigma$ is $\bink nd$ and
$\frac1n\ln\bink nd\sim h(d/n)$ by Stirling.
\qed\end{proof}

\begin{lemma}\label{Lemma_gmax}
The function $g$ from \Cor~\ref{Cor_cE} takes its unique maximum in the interval $(1/2-\gamma,1/2+\gamma)$ at $1/2$,
	and $g''(1/2)<0$.
\end{lemma}
\begin{proof}
Let $\tau$ be equitable and at distance $\alpha n$ from $\sigma$.
Moreover, let $X_1(\alpha)$ be the number of edges of $H_k(n,m_1,m_2,\sigma)$ that are critical under both $\sigma,\tau$.
In addition, let $X_2(\alpha)$ be the number of edges that are critical under $\tau$ but not under $\sigma$.
As $H_k(n,m_1,m_2,\sigma)$ consists of two independent `portions' of random edges,
namely $m_1$ that are critical under $\sigma$ and another $m_2$ that are not,
 $X_1(\alpha),X_2(\alpha)$ are independent.
 Furthermore,
	$$X_1(\alpha)\sim\Bin(m_1,q_1(\alpha)),\qquad X_2(\alpha)\sim\Bin(m-m_1,q_2(\alpha)),$$
where $q_1(\alpha)=\frac{u_1}{1-v_1}(\alpha)$, $q_2(\alpha)=\frac{u_2}{1-v_2}(\alpha)$.
Let
	$$p(\alpha)=\pr\brk{X_1(\alpha)+X_2(\alpha)=m_1},\qquad p(\alpha,x_1,x_2)=\pr\brk{X_1(\alpha)=x_1\wedge X_2(\alpha)=x_2},$$
so that
	$$p(\alpha)=\sum_{x_1,x_2:x_1+x_2=m_1}p(\alpha,x_1,x_2).$$

Let us first investigate the point $\alpha=1/2$.
As $q_1(1/2)=q_2(1/2)$, we have
	$$(X_1+X_2)(1/2)\sim\Bin(m_1,q).$$
with
	$$q=q_1(1/2)=q_2(1/2)=\frac{2k}{2^k-2}=\frac{k}{2^{k-1}-1}.$$
Using 
\Lem~\ref{Lemma_binlargedev}, we can compute $p(1/2)$ directly:
	\begin{eqnarray*}
	p(1/2)&=&\bink m{m_1}q^{m_1}(1-q)^{m-m_1}\\
		&=&\Theta(n^{-1/2})\cdot\exp\brk{-mq\cdot\varphi\bcfr{m_1-mq}{mq}-(1-q)m\cdot\varphi\bcfr{mq-m_1}{m(1-q)}},
	\end{eqnarray*}
where $\varphi(x)=(1+x)\ln(1+x)-x$.
Hence,
	\begin{eqnarray*}
	p(1/2)
		&=&\Theta(n^{-1/2})\cdot\exp\brk{-m\bc{q\cdot\varphi\bc{\beta}-(1-q)\varphi\bcfr{\beta q}{q-1}}}.
	\end{eqnarray*}
To proceed, we need to decompose this expression according to the individual contributions of $X_1(1/2),X_2(1/2)$.
Let $x_1,x_2$ be such that $x_1+x_2=m_1$.
Since $X_1(1/2),X_2(1/2)$ are independent, we have
	\begin{eqnarray*}
	p(1/2,x_1,x_2)&=&\pr\brk{X_1=x_1\wedge X_2=x_2}\\
		&=&\pr\brk{\Bin(m_1,q)=x_1}\cdot\pr\brk{\Bin(m_2,q)=x_2}\\
		&=&\Theta(n^{-1})\exp\brk{-qm_1\varphi\bcfr{x_1-m_1q}{m_1q}-(1-q)m_1\varphi\bcfr{m_1q-x_1}{(1-q)m_1}}\\
			&&\qquad\qquad\cdot\exp\brk{-qm_2\varphi\bcfr{x_2-m_2q}{m_2q}-(1-q)m_2\varphi\bcfr{m_2q-x_2}{(1-q)m_2}},
	\end{eqnarray*}
and
	$$p(1/2)=\sum_{x_1+x_2=m_1}p(1/2,x_1,x_2).$$
Similarly, for general $\alpha$ we have
	\begin{eqnarray*}
	p(\alpha,x_1,x_2)
		&=&\Theta(n^{-1})\exp\brk{-q_1(\alpha)m_1\varphi\bcfr{x_1-m_1q_1(\alpha)}{m_1q_1(\alpha)}-(1-q_1(\alpha))m_1
				\varphi\bcfr{m_1q_1(\alpha)-x_1}{(1-q_1(\alpha))m_1}}\\
			&&\qquad\cdot\exp\brk{-q_2(\alpha)m_2\varphi\bcfr{x_2-m_2q_2(\alpha)}{m_2q_2(\alpha)}-(1-q_2(\alpha))
					m_2\varphi\bcfr{m_2q_2(\alpha)-x_2}{(1-q_2(\alpha))m_2}},
	\end{eqnarray*}
and
	$$p(\alpha)=\sum_{x_1+x_2=m_1}p(\alpha,x_1,x_2).$$

As $u_i(1-\alpha)=u_i(\alpha)$ for $\alpha\in(0,1)$, we have $u_i'(1/2)=v_i'(1/2)=0$ for $i=1,2$.
Moreover, a direct calculation shows that
	\begin{eqnarray*}
	u_1''(1/2)&=&O(k^2/2^k),\\
	u_2''(1/2)&=&O(k^3/4^k).
	\end{eqnarray*}
Hence, by the chain rule,
	$$\frac{\partial^2}{\partial\alpha^2}\bc{\ln p(\alpha,x_1,x_2)}\big|_{\alpha=1/2}=O(k^3/2^k),$$
whence
	\begin{equation}\label{eqTaylorp}
	\ln\frac{p(1/2-\delta,x_1,x_2)}{p(1/2,x_1,x_2)}=\delta^2\cdot O(k^3/2^k)+O(\delta^3).
	\end{equation}
Furthermore, as $v_1''(1/2)=O(k^2/2^{k})$, $v_2''=O(k^3/4^k)$, we also obtain
	\begin{equation}\label{eqTaylorv}
	\ln\frac{(1-v_1(1/2-\delta))^{m_1}(1-v_2(1/2-\delta))^{m_2}}{(1-v_1(1/2-\delta))^{m_1}(1-v_2(1/2-\delta))^{m_2}}=
		\delta^2\cdot O(k^3/2^k)+O(\delta^3).
	\end{equation}
As the derivatives of the entropy are
	\begin{eqnarray*}
	h'(\alpha)&=&-\ln\alpha+\ln(1-\alpha),\\
	h''(\alpha)&=&-\frac1\alpha-\frac1{1-\alpha},
	\end{eqnarray*}
(\ref{eqTaylorp}) and~(\ref{eqTaylorv}) yield
	\begin{equation}\label{eqTaylorE}
	\cE(1/2-\delta)=-4\delta^2+\delta^2\cdot O(k^3/2^k)+O(\delta^3).
	\end{equation}
Finally, (\ref{eqTaylorE}) shows that $g$ takes its unique maximum in  $(1/2-\gamma,1/2+\gamma)$ at $1/2$,
and $g''(1/2)<0$.
\qed\end{proof}

Combining (\ref{eqgoodlocalbeta}), \Lem~\ref{Lemma_gintermediate}, and \Lem~\ref{Lemma_gmax},
we obtain~(\ref{eqgooddist}).
This completes the proof of \Prop~\ref{Prop_critsecond}.

\noindent\emph{Proof of \Thm~\ref{Thm_cond}.}
Let $r=r_{crit}$, $\delta=\delta_k,\beta=\beta_k$ be as in \Prop~\ref{Prop_critical}.
Then by \Prop~\ref{Prop_critsecond}, the probability that $H_k(n,m)$ is $2$-colorable is bounded from below by a positive constant.
As $2$-colorability in $H_k(n,m)$ has a sharp threshold, this implies that $r_{col}>r_{crit}$.
This proves the first assertion of \Thm~\ref{Thm_cond}.

To prove the second assertion, fix $r=r_{crit}$.
Then the second part of \Prop~\ref{Prop_critical} implies that \whp
	\begin{equation}\label{eqZZ}
	\frac1n\ln Z(H_k(n,m,\sigma))>\frac1n\ln\Erw\brk{Z}+\delta n.
	\end{equation}
However, there is no obvious way to derive the second assertion in \Thm~\ref{Thm_cond} directly from~(\ref{eqZZ}),
because it is not clear \emph{a priori} that the random variable $\frac1n\ln Z(H_k(n,m,\sigma))$ is tightly concentrated.
Therefore, we will replace it by another random variable for which concentration is easy to show.
Namely, for any $b>0$ we let
	$$\cZ_b=\sum_{\sigma\in\cbc{0,1}^n}\exp(-b w(\sigma)),$$
where $w(\sigma)$ is the number of monochromatic edges under $\sigma$.
(The above random variable is called the partition function at inverse temperature $b$.)
The random variable $\frac1n\ln\cZ_b$ satisfies a Lipschitz condition:
	either adding or removing a single edge can change the value of $\frac1n\ln\cZ_b$ by at most $b$.
Therefore, Azuma's inequality implies that in both the random hypergraph $H_k(n,m)$ and in the planted model $H_k(n,m,\sigma)$ we have
	\begin{equation}\label{eqZZ2}
	\pr\brk{\abs{\ln\cZ_b-\Erw\ln\cZ_b}>y}\leq2\exp\brk{-\frac{y^2}{2m}}.
	\end{equation}

We are going to derive an upper bound on $\Erw\ln\cZ_b(H_k(n,m))$.
To this end, let $S_\mu$ denote the number of $\sigma\in\cbc{0,1}^n$ with $w(\sigma)=\mu$ in $H_k(n,m)$.
Then
	\begin{equation}\label{eqZZ3}
	\cZ_b=\sum_{\mu=0}^m\exp(-b\mu)\cdot|S_\mu|.
	\end{equation}
Furthermore, letting $\mu=\gamma m$ for some small $\gamma>0$ and using %
\Lem~\ref{Lemma_binlargedev}, we obtain
	\begin{eqnarray}\nonumber
	\frac1n\ln\Erw S_\mu&\sim&\ln2+\frac1n\ln\pr\brk{\Bin(m,2^{1-k})=\mu}\\
		&=&\ln2-\frac{r}{2^{k-1}}\brk{\varphi(1-2^{k-1}\gamma)+(2^{k-1}-1)\varphi\bc{-\frac{2^{k-1}\gamma-1}{2^{k-1}-1}}}
		\label{eqZZ4}
	\end{eqnarray}
where $\varphi(x)=(1+x)\ln(1+x)-x$.
Plugging~(\ref{eqZZ4}) into~(\ref{eqZZ3}), we see that
	\begin{equation}\label{eqZZ5}
	\frac1n\Erw\ln\cZ_b\leq\frac1n\ln\Erw Z(H_k(n,m))+\eps_b,
	\end{equation}
where $\eps_b\ra0$ as $b\ra\infty$.
Intuitively, this mirrors the fact that the partition function is dominated by assignments that violate an $o_b(1)$-fraction
of all clauses as $b\ra\infty$.
From now on, fix $b$ large enough so that $\eps_b<\delta/4$.
Thus, (\ref{eqZZ2}) and~(\ref{eqZZ5}) imply that
	\begin{equation}\label{eqZZ6}
	\pr_{H_k(n,m)}\brk{\frac1n\ln\cZ_b>\frac1n\ln\Erw\brk{Z(H_k(n,m))}+\frac{\delta}3}=o(1).
	\end{equation}	

We will contrast~(\ref{eqZZ6}) with the situation in the planted model.
Let $\xi>0$ be sufficiently small and let $\tau\in\cbc{0,1}^n$ be such that
	$\abs{|\tau^{-1}(0)|-|\tau^{-1}(1)|}<\xi n$.
Then there is an equitable $\sigma$ such that $\dist(\sigma,\tau)\leq\xi n$.
In $H_k(n,m,\sigma)$ the number of edges that are monochromatic under $\tau$
has a binomial distribution with mean $\leq k\xi m$.
Therefore, we obtain
	$$\frac1n\Erw\ln\cZ_b(H_k(n,m,\tau))\geq\frac1n\Erw\ln\cZ_b(H_k(n,m,\sigma))-km\xi-o(1).$$
Combining this with~(\ref{eqZZ}) and choosing $\xi>0$ sufficiently small, we thus get
	$$\frac1n\Erw\ln\cZ_b(H_k(n,m,\tau))\geq\frac1n\ln\Erw Z(H_k(n,m))+2\delta/3.$$
Hence, (\ref{eqZZ2}) yields that
	\begin{equation}\label{eqZZ7}
	\pr\brk{\frac1n\ln\cZ_b(H_k(n,m,\tau))<\frac1n\ln\Erw Z(H_k(n,m))+\delta/2}\leq\exp(-\xi''n),
	\end{equation}	
with $\xi''>0$.

To complete the proof, consider the set
	$\Lambda$ of all pairs $(H,\tau)$ of hypergraphs $H$ on $V=\cbc{1,\ldots,n}$ with $m$ edges and $2$-colorings $\tau$.
Let
	\begin{eqnarray*}
	\Lambda'&=&\cbc{(H,\tau)\in\Lambda:\frac1n\ln\cZ_b(H)<\frac1n\ln\Erw Z(H_k(n,m))+\delta/2},\\
	\Lambda''&=&\cbc{(H,\tau)\in\Lambda:\abs{\abs{\tau^{-1}(1)}-\abs{\tau^{-1}(0)}}\geq\xi n}.
	\end{eqnarray*}
Clearly, (\ref{eqZZ7}) shows that
	$$\abs{\Lambda'\setminus\Lambda''}\leq\exp(-\xi''n)\abs\Lambda.$$
Furthermore, since the number of hypergraphs $H$ for which $\tau\in\cbc{0,1}^n$ is a $2$-coloring is maximized for equitable
$\tau$, we have
	$\abs{\Lambda''}\leq\exp(-\xi'''n)\abs\Lambda,$
with $\xi'''>0$ sufficiently small.
Hence,
	\begin{equation}\label{eqZZ8}
	\abs{\Lambda'}\leq2\exp(-\xi'''n)\abs\Lambda.
	\end{equation}

Now, suppose that $\alpha_1,\alpha_2$ are such that
	$$\pr\brk{\frac1n\ln Z(H_k(n,m))\geq\frac1n\ln\Erw Z(H_k(n,m))-\alpha_1n}\geq\alpha_2.$$
Since $H_k(n,m)$ is uniformly distributed over all $\bink{\bink nk}m$ hypergraphs with $m$ edges, we obtain
	\begin{equation}\label{eqZZ9}
	\abs{\Lambda'}\geq\alpha_2\bink{\bink nk}m\Erw\brk{Z(H_k(n,m))}\exp(-\alpha_1n)\geq
		\alpha_2\exp(-\alpha_1n)\cdot\abs{\Lambda}.
	\end{equation}
Setting $\alpha_1=\xi'''/2$ and comparing~(\ref{eqZZ8}) with~(\ref{eqZZ9}), we see that necessarily $\alpha_2=o(1)$.
This proves~(\ref{eqcond}) in the case that $r=r_{crit}$.

Finally, consider any density $r_{crit}<r<r_{col}$.
We generate the random hypergraph $H_k(n,m)$ in two `portions' $H_1$ and $H_2$.
Namely, letting $m_1=r_{cond}n$ and $m_2=(r-r_{cond})n$, we let $H_1=H_k(n,m_1)$.
Then $H_2$ is simply obtained by adding another $m_2$ random edges to $H_1$.
By the above, we know that \whp
	$$Z(H_1)\leq \Erw\brk{Z(H_1)}\cdot\exp(-\Omega(n)).$$
Furthermore, a new random edge is bichromatic under a $2$-coloring of $H_1$ with probability $1-2^{1-k}$,
we have
	$$\Erw\brk{Z(H_2)|H_1}\leq Z(H_1)\cdot(1-2^{1-k})^{m_2}.$$
Thus, \whp
	$$Z(H_k(n,m))=Z(H_2)\leq\Erw\brk{Z(H_1)}\cdot\exp(-\Omega(n))(1-2^{1-k})^{m_2}
		=\Erw\brk{Z(H_k(n,m))}\exp(-\Omega(n)),$$
as claimed.
\qed

\subsection{Proof of \Lem~\ref{Lemma_psi}}\label{Sec_psi}

Since $\psi(1-x)=\psi(x)$, we only need to work with $x\leq 1/2$.
Let $r=(2^{k-1}-1)(c/2^k+\ln2)$ for $|c|\leq4$.
Let $h(x)=-x\ln x-(1-x)\ln(1-x)$.
Then
	$$\psi(x)\leq h(x)-\frac{r}{2^{k-1}-1}(1-x^k-(1-x)^k)
		\leq h(x)-(c/2^k+\ln2)(1-x^k-(1-x)^k).$$

Suppose that $x>2^{-k/2}$ but $x<1/(1.01k)$.
Then
	\begin{eqnarray}\nonumber
	\psi(x)&\leq&x(1-\ln x)-(c/2^k+\ln2)\bc{1-(1-x)^k}+2^{-k}\\
		&\leq&x(1-\ln x)-(c/2^k+\ln2)\bc{kx-(kx)^2}+2^{-k}\nonumber\\
		&\leq&x\brk{1-\ln x-k(1-kx)\ln2}+2^{3-k}<-2^{3-k}\leq-\psi(1/2),
			\label{eqpsi1}
	\end{eqnarray}
provided that $k\geq k_0$ is large enough.
Furthermore, for $1/(1.01k)<x<0.49$ we have
	\begin{eqnarray}\nonumber
	\psi(x)&\leq&h(x)-(c/2^k+\ln2)\bc{1-(1-x)^k}+2^{-k}\\
		&\leq&h(x)-(c/2^k+\ln2)(\exp(-kx)-1)+2^{-k}\nonumber\\
		&\leq&h(x)-(\exp(-kx)-1)\ln2+2^{3-k}<-2^{3-k}\leq-\psi(1/2),
						\label{eqpsi2}
	\end{eqnarray}
again for $k\geq k_0$ large enough.

Finally, around $x=1/2$ we can expand $\psi$ as follows.
Since $\psi(1-x)=\psi(x)$, it is clear that $\psi'(1/2)=0$.
Furthermore, $\psi''(1/2)=-4+o_k(1)$,
and $\psi'''(1/2)\leq h'''(1/2)+o_k(1)=o_k(1)$.
Therefore, for $k\geq k_0$ large enough we can expand $\psi$ around $1/2$ as
	\begin{equation}		\label{eqpsi3}
	\psi\bc{\frac12-\delta}=\psi(1/2)-(4+o_k(1))\delta^2+O(\delta^3).
	\end{equation}
Thus, the lemma follows from (\ref{eqpsi1})--(\ref{eqpsi3}).

\section{The local cluster:
	proof of  \Prop s~\ref{Prop_local} and~\ref{Prop_critical}}\label{Sec_localCluster}

\subsection{Outline}

In this section we prove \Prop s~\ref{Prop_local} and~\ref{Prop_critical}.
Fix an equitable $2$-coloring $\sigma:V\ra\cbc{0,1}$ and
recall that an edge $e$ of a hypergraph $H$ is \emph{critical} under $\sigma$
if there is a color $i\in\cbc{0,1}$ and a vertex $v\in E$ such that $\sigma(v)=i$ and $\sigma(w)=1-i$ for all $w\in e\setminus\cbc v$.
In this case, we say that $v$ \emph{supports} the edge $e$ (under $\sigma$).

We are going to study the size of the local cluster in the $H_k(n,m_1,m_2,\sigma)$ model from \Sec~\ref{Sec_beyond}:
\begin{itemize}
\item Choose a set $E_1$ of $m_1$ edges that are critical with respect to $\sigma$ uniformly at random.
\item Choose a set $E_2$ of $m_2$ edges that are bichromatic under $\sigma$ but not critical uniformly at random.
\item Let $H_k(n,m_1,m_2,\sigma)=(V,E_1\cup E_2)$.
\end{itemize}
We are going to expose the edges of $H_k(n,m_1,m_2,\sigma)$ in two portions:
let $H_1$ contain the $m_1$ critical edges, and let $H_2$ contain the rest.
Let $\lambda=m_1/n$ be the expected number of edges that any one vertex supports.

We will need the following simple expansion property of the random hypergraph $H_1$.

\begin{lemma}\label{Lemma_neighborhood}
Let $\zeta<1/3$.
\Whp\ $H_1$ has the following property.
Suppose that $S\subset V$ has size $|S|=\zeta n$.
Then \whp\ the total number of edges supported by vertices in $S$
is bounded by $\zeta(\eul^3\lambda-\ln\zeta)n$.
\end{lemma}
\begin{proof}
We use a first moment argument.
Let $\xi=\eul^3\lambda-\ln\zeta$ and $\mu=\xi\zeta$.
The probability that there is a set $S$ of size $\zeta n$ that supports a total of $\mu n$ edges is bounded by
	\begin{eqnarray*}
	\bink{n}{\zeta n}\bink{m_1}{\mu n}\zeta^{\mu n}&\leq&
		\brk{\bcfr{\eul}{\zeta}^\zeta\bcfr{\eul\lambda\zeta}{\mu}^\mu}^n
			=\brk{\frac{\eul}{\zeta}\bcfr{\eul\lambda\zeta}{\mu}^\xi}^{\zeta n}\\
	&\leq&\brk{\frac{\eul}{\zeta}\bcfr{\eul\lambda}{\xi}^\xi}^{\zeta n}\leq\brk{\eul\zeta}^{\zeta n}=o(1),
	\end{eqnarray*}
by our choice of $\xi$ and because $\zeta<1/3$.
\qed\end{proof}

\begin{lemma}\label{Lemma_Poisson}
Let $l\geq0$ be fixed.
\Whp\ the number of vertices that support precisely $l$ edges is
	$$(1+o(1))n\cdot\frac{\lambda^l}{l!\exp(\lambda)}$$
\end{lemma}
\begin{proof}
The number of edges that any one vertex supports is binomial with mean $\lambda$.
Hence, the Poisson approximation to the binomial distribution shows that the probability that some vertex $v$
supports precisely $l$ edges is $(1+o(1))\frac{\lambda^l}{l!\exp(\lambda)}$.
In effect, letting $X_l$ be the number of vertices with this property, we see that
	$$\Erw X_l=(1+o(1))n\cdot\frac{\lambda^l}{l!\exp(\lambda)}.$$
Furthermore, $X_l$ satisfies a Lipschitz condition: adding or removing a single edge
	can change the value of $X_l$ by at most one.
Therefore, Azuma's inequality shows that $X_l=(1+o(1))n\cdot\frac{\lambda^l}{l!\exp(\lambda)}$ \whp
\qed\end{proof}

In particular, \Lem~\ref{Lemma_Poisson} shows that the total number of vertices that do not support any edges is $(1+o(1))\exp(-\lambda)n$ \whp\
Now, consider the following construction of a set $U\subset V$.

\begin{enumerate}
\item Initially, let $U$ consist of all vertices that do not support any edges.
\item While there is a vertex $v\not\in U$ that does not support an edge that does not contain a vertex from $U$, add $v$ to $U$.
\end{enumerate}
The above is an adaptation of the `whitening process' from~\cite{fede} to random hypergraph $2$-coloring.

Let $H_U$ be the hypergraph with vertex set $U$ and edge set
	$$\cbc{e\cap U:e\in E(H_1),|e\cap U|\geq2}.$$
In general, this is going to be a non-uniform hypergraph.

\begin{proposition}\label{Prop_ZU}
\Whp\ the set $U$ has size
	$$|U|=n\brk{\exp(-\lambda)+\lambda(k-1)\exp(-2\lambda)+O(7.1^{-k})}$$
and enjoys the following properties.
\begin{description}
\item[U1.] The set $S_0\subset U$ of variables that do not support a clause has size has size
		$|S_0|=(1+o(1))n\exp(-\lambda)$.
\item[U2.] There is a set $S_1$ of size
		$$(1+o(1))n\brk{\lambda(k-1)\exp(-2\lambda)+O(7^{-k})}$$
	such that all vertices in $S_1$ support exactly one edge
	that contains precisely one other vertex from $U$, which indeed belongs to $S_0$.
\item[U3.] Apart from the edges resulting from {\bf U2}, $H_U$ contains no more than $n O(7.1^{-k})$ further edges.
\end{description}
\end{proposition}
We defer the proof of \Prop~\ref{Prop_ZU} to \Sec~\ref{Sec_ZU}.

We say that $R\subset V$ is \emph{rigid} if for any $2$-coloring $\tau$ of $H_1$ such
that $\tau(v)\neq\sigma(v)$ we have
	$$\abs{\cbc{v\in R:\tau(v)\neq\sigma(v)}}\geq n/k^3.$$
In \Sec~\ref{Sec_Z1} we will prove the following.

\begin{proposition}\label{Prop_Z1}
\Whp\ there is a rigid set $R\subset V\setminus U$ of size $|R|\sim|V\setminus U|$.
\end{proposition}

We now have sufficient information about the random hypergraph
$H_k(n,m_1,m_2,\sigma)=H_1\cup H_2$ to prove \Prop s~\ref{Prop_local} and~\ref{Prop_critical}.

\noindent\emph{Proof of \Prop~\ref{Prop_local}.}
\Prop~\ref{Prop_local} deals with the random hypergraph $H_k(n,m,\sigma)$,
in which the number of critical edges has a binomial distribution $\Bin(m,k/(2^{k-1}-1))$.
Hence, by Chernoff bounds the number of critical edges is
$(1+o(1))m kr/(2^{k-1}-1)=(1+o(1))\lambda n$ \whp,
with $\lambda=kr/(2^{k-1}-1)$.
Thus, to study $H_k(n,m,\sigma)$ it suffices to investigate $H_k(n,m_1,m_2,\sigma)$ with $m_1\sim\lambda n$
and $m_2\sim m-\lambda n$.

To prove \Prop~\ref{Prop_local} we merely need to derive an \emph{upper} bound on the size $\abs{\cC(\sigma)}$ of the local cluster.
Thus, it suffices to bound the size of the local cluster
	$$\cC^1(\sigma)=\cbc{\tau:\dist(\sigma,\tau)\leq2^{-k/2}n,\quad\tau\mbox{ is a $2$-coloring of $H_1$}}$$
of $H_1$.
By \Prop~\ref{Prop_Z1}, we have \whp
	$$\frac1n\ln\abs{\cC(\sigma)}\leq\frac1n\ln\abs{\cC^1(\sigma)}=
		\frac1n\ln Z(H_U).$$
Hence, we just need to bound the number $Z(H_U)$ of $2$-colorings of $H_U$.
By \Prop~\ref{Prop_ZU} we may assume that $H_U$ has the properties {\bf U1--U3}.

If this is indeed the case, we can estimate $\ln Z_U$ as follows.
Let $H_U'$ be the hypergraph obtained from $H_U$ by omitting all edges that are incident with a vertex from $U\setminus(S_0\cup S_1)$.
Then each edge of $H_U'$ has size $2$ and contains precisely one vertex from $S_1$ and one vertex from $S_0$.
Moreover, each vertex from $S_1$ is incident with exactly one such edge, and indeed supports this edge under $\sigma$.
Hence, $H_U'$ is just a collection of stars in which all non-isolated vertices in $S_1$ are leaves,
and therefore the total number of $2$-colorings of $H_U'$ is simply
equal to $2^{|S_0|}$.
Thus,
	$$\frac1n\ln\abs{\cC(\sigma)}\leq\frac1n\ln Z(H_U)\leq \frac1n\ln Z(H_U')\leq\frac{|S_0|}n\ln2\leq(1+o(1))\exp(-\lambda)\ln2.$$
A straightforward computation shows that this is indeed less than $\frac1n\ln\Erw\brk{Z_e(H_k(n,m))}$ if
$r<(2^{k-1}-1)\ln2$.
\qed

\noindent\emph{Proof of \Prop~\ref{Prop_critical}.}
We start by obtaining an upper bound on the size of $\cC(\sigma)$.
Let  $\lambda=(1+\beta)kr/(2^{k-1}-1)$ for some $\beta\leq1/k$.
We first study the size of the local cluster $\cC^1(\sigma)$ in $H_1$.
By the same argument as in the proof of \Prop~\ref{Prop_local} above, \whp\ we have
	$$\frac1n\ln\abs{\cC^1(\sigma)}\leq\frac1n\ln Z(H_U')\leq(1+o(1))\exp(-\lambda)\ln2,$$
where $H_U'$ is a collection of stars as above.
While clearly $\frac1n\ln\abs{\cC(\sigma)}\leq\frac1n\ln\abs{\cC^1(\sigma)}$,
we need a slightly tighter estimate of $\abs{\cC(\sigma)}$.

To obtain this estimate, we need to take the edges of $H_2$ into consideration.
Let $E_2'$ consist of all edges $e\in H_2$ that contain precisely two vertices from $S_0\setminus N(S_1)$
and in which all vertices in $V\setminus U$ have the same color under $\sigma$.
Since $H_2$ is independent of $H_1$, the number of these edges is binomially distributed with mean
	$$\frac{|S_0\setminus N(S_1)|^2}{n^2}\cdot\frac{\bink k2}{2^{k-1}-1}\cdot n\cdot m_2\geq n\bink k2\exp(-2\lambda)\ln2=\mu_2.$$
By Chernoff bounds, we indeed have $|E_2'|\geq(1-o(1))\mu_2$ \whp\
Furthermore, the expected number of vertices in $S_0$ that are incident with two edges from $E_2'$
is $\leq O(k^4\exp(-3\lambda))$; as this number satisfies a Lipschitz condition, it is concentrated by Azuma's inequality.
Hence, \whp\ $E_2'$ contains a subset $E_2''$ of size
	$$|E_2''|/n\geq\bink k2\exp(-2\lambda)\ln2-O(k^4/8^k)$$
such that $E_2''$ induces a matching in $S_0$.
By construction, this matching is disjoint from $H_U'$.
Hence,  \whp\
	\begin{equation}\label{eqlocalupper}
	\frac1n\ln\abs{\cC(\sigma)}\leq\frac1n\ln\abs{\cC^1(\sigma)}-|E_2''|\ln2\leq
			\exp(-\lambda)\brk{1-\bink k2\exp(-\lambda)\ln2}\ln2+O(k^4/8^k).
	\end{equation}

To derive a matching lower bound, notice that \Prop~\ref{Prop_ZU} implies that
all but $O(7.1^{-k})n$ edges of $H_1$ belong to the matching $H_U'$ \whp\
Let $F_1$ be the set of all vertices that are reachable from the edges in $H_1\setminus H_U'$.
Then $|F_1|\leq4|H_1\setminus H_U'|\leq O(7.1^{-k})n$ \whp\
While we cannot say much about the entropy of the vertices in $F_1$, it is clear
that $H_U- F_1$ is just a matching from $S_1\setminus F$ to $S_0\setminus F$.
Therefore, \whp\
	$$\frac1n\ln\abs{\cC^1(\sigma)}\geq|S_0\setminus F|\ln2\geq(\exp(-\lambda)-O(7.1^{-k}))\ln2.$$
Let $E_3'$ be the set of all edges $e\in H_2$ that contain at least three vertices from $U$
such that all vertices in $e\setminus U$ have the same color under $\sigma$.
Then
	$\Erw|E_3'|\leq O(k^3/2^k)\cdot (|U|/n)^3 m\leq O(k^3/8^k)n$.
Let $F_3$ be the set of all vertices in $H_U$ that are reachable from $\cbc{v\in U:\exists e\in E_3':v\in e}$.
Since $|E_3'|$ is binomially distributed, we have $|F_3|\leq O(k^4/8^k)n$ \whp\
Furthermore, let $E_2'$ be as above.
Let $F_2'$ be the set of all vertices in $N(S_1)\cup U\setminus(S_0\cup F_1\cup F_3)$
that are incident with an edge of $E_2'$.
Then $|F_2'|$ is binomially distributed with mean $\leq O(k^2/2^k)\exp(-\lambda)|U\setminus S_0|m\leq O(k^3/8^k)n$
	(by \Prop~\ref{Prop_ZU}),
and thus \whp\ $|F_2'|\leq O(k^3/8^k)n$ by Chernoff bounds.
In addition, let $F_2''$ be the set of all vertices in $S_0$ that are incident with at least two edges from $E_2'$.
As we saw above, $F_2''\leq O(k^4/8^k)n$ \whp\
Let $F_2$ be the set of all vertices in $H_U'$ that are reachable from $F_2'\cup F_2''$;
since $H_U'$ is a matching, we have $|F_2|\leq2|F_2'\cup F_2''|=O(k^4/8^k)n$ \whp\
Finally, let $E_2''$ be the set of all edges in $E_2'$ that do not contain a vertex from $F_2$.
Then $|E_2''|/n\leq\bink k2\exp(-2\lambda)\ln2-O(k^4/8^k)$ \whp\

Now, $E_2''$ and $H_U'$ simply induce a matching on $(S_0\cup S_1)\setminus(F_1\cup F_2\cup F_3)$,
and this matching is disconnected from all other edges of $H_1\cup H_2$ that are not already $2$-colored given the
colors assigned to the vertices in $V\setminus U$.
Hence, \whp\ the number of $2$-colorings is at least
	\begin{eqnarray}\nonumber
	\frac1n\ln\abs{\cC(\sigma)}&\geq&
		\frac1n\ln\abs{\cC^1(\sigma)}-|E_2''|\ln2-|F_1\cup F_2\cup F_2|\ln 2\\
		&\geq&\exp(-\lambda)\brk{1-\bink k2\exp(-\lambda)\ln2}\ln2-O(7.1^{-k}).
		\label{eqlocallower}
	\end{eqnarray}

To prove the first claim, we need to combine~(\ref{eqlocalupper}) and~(\ref{eqlocallower})
with a lower bound on the expected number of $(1+\beta)$-critical $2$-colorings.
Let $q=k/(2^{k-1}-1)$.
The probability $\eta_{1+\beta}$ that an equitable $\sigma$ is  a $(1+\beta)$-critical $2$-coloring of $H_k(n,m)$ satisfies
	$$\ln\eta_{1+\beta}\sim m\ln(1-2^{1-k})+\ln\pr\brk{\Bin(m,q)=(1+\beta)qm}.$$
Indeed, the first summand accounts for the probability that $\sigma$ is a $2$-coloring,
and the second summand is the probability that given that $\sigma$ is a $2$-coloring, the number
of critical edges equals $(1+\beta)qm$.
By the \Lem~\ref{Lemma_binlargedev},
for sufficiently small $\beta>0$ we have
	$$\frac1n\ln\pr\brk{\Bin(m,q)=(1+\beta)qm}\geq-\frac{\beta^2qm}{n}\geq-k\beta^2.$$
Hence,
	$$\frac1n\ln\Erw\brk{Z_{1+\beta}}\geq\frac1n\ln\Erw Z-k\beta^2.$$
If $r=2^{k-1}\ln2-c$, then a direct computation shows that
	\begin{equation}\label{eqErwZ}
	\frac1n\ln\Erw\brk{Z}=\ln2+r\ln\bc{1-2^{1-k}}\geq\frac{(2c-\ln2)}{2^k}-O(4^{-k}).
	\end{equation}
Consequently,
	\begin{equation}\label{eqheavylower}
	\frac1n\ln\Erw\brk{Z_{1+\beta}}\geq\frac{(2c-\ln2)}{2^k}-O(4^{-k})-k\beta^2.
	\end{equation}

Choose $c$ (and thus $r$) such that with $\lambda_0=kr/(2^{k-1}-1)$ we have
	\begin{equation}\label{eqheavyXi}
	\Xi=\exp(-\lambda_0)\brk{1-\bink k2\exp(-\lambda_0)\ln2}\ln2-7^{-k}=\frac1n\ln\Erw Z+16^{-k}.
	\end{equation}
A straight computation using~(\ref{eqErwZ}) shows that $c=\ln2+o_k(1)$.
Furthermore, (\ref{eqlocallower})  and~(\ref{eqheavyXi}) show that for this $r$
\whp\ in the planted model $H_k(n,m,\sigma)$ the local cluster $\cC(\sigma)$ has size
$|\cC(\sigma)|>\exp(\Omega(n))\Erw Z$.
Let
	$$f(\beta)=\exp(-(1+\beta)\lambda_0)\brk{1-\bink k2\exp(-(1+\beta)\lambda_0)\ln2}\ln2.$$
Expanding $f(\cdot)$ around $\beta=0$, we find that
	$$f(\beta)-f(0)=-\beta(\exp(-\lambda_0)\ln2+O(k^24^{-k}))+O(\beta^2)/2^k.$$
Hence, (\ref{eqheavylower}) implies that for $\beta^*=3^{-k}$ we get
	$$f(\beta^*)+7^{-k}<\frac1n\ln\Erw\brk{Z_{1+\beta^*}}.$$
Further, (\ref{eqlocalupper}) implies that with $m_1=(1+\beta^*)\lambda_0n$, $m_2=m-m_1$
in $H_k(n,m_1,m_2,\sigma)$ \whp\ the local cluster size satisfies
	$\frac1n\ln\abs{\cC(\sigma)}<\frac1n\ln\Erw\brk{Z_{1+\beta^*}}$.
This means that $r,\beta^*$ as above satisfy the conditions in \Prop~\ref{Prop_critical}.
\qed

\subsection{Proof of \Prop~\ref{Prop_ZU}}\label{Sec_ZU}

Let $U_0$ be the set of all vertices that do not support any edge.
Then \whp\ $|U_0|\sim n\exp(-\lambda)$ by \Lem~\ref{Lemma_Poisson}.
For each vertex $v$ let $s(v)$ be the number of edges that $v$ supports.
Let $U_1$ be the set of all vertices $v$ with $s(v)\geq1$ such that all edges supported by $v$ contain a vertex from $U_0$.

\begin{lemma}\label{Lemma_U0}
\Whp\ we have the following.
\begin{enumerate}
\item The number of vertices $v$ with $s(v)=1$ such that the edge $e$ supported by $v$
		contains exactly one vertex from $U_0$ is
			$$n\brk{\lambda(k-1)\exp(-2\lambda)+O(7.3^{-k})}.$$
\item The number of vertices $v$ with $s(v)=1$ such that the edge $e$ supported by $v$
		contains more than one vertex from $(U_0\cup U_1)\setminus\cbc v$ is $n\cdot O(7.3^{-k})$.
\item The number of vertices $v$ with $s(v)>1$ such that all edges $e$ supported by $v$
		contain a vertex from $U_0$ is bounded by $n\cdot O(7.3^{-k})$.
\end{enumerate}
\end{lemma}
\begin{proof}
Let $X$ be the number of vertices as in 1.
By \Lem~\ref{Lemma_Poisson}, the number of vertices $v$ with $s(v)=1$ is $(1+o(1))\lambda\exp(-\lambda)n$ \whp\
Furthermore, given that $v$ satisfies $s(v)=1$, the $k-1$ other vertices in the unique edge $e$ that $v$ supports are uniformly distributed
over the opposite color class.
Hence, again by \Lem~\ref{Lemma_Poisson}, the number of non-supporting vertices amongst these $k-1$ vertices
has a binomial distribution $\Bin(k-1,(1+o(1))\exp(-\lambda))$ \whp\
In this case, the probability that exactly one of the $k-1$ other vertices is non-supporting is
	$(k-1)\exp(-\lambda)+O(\exp(-2\lambda)).$
Hence, we see that
	$$\Erw X=(1+o(1))\lambda\exp(-\lambda)\cdot\brk{(k-1)\exp(-\lambda)+O(\exp(-2\lambda))}
			=n\brk{\lambda(k-1)\exp(-2\lambda)+O(7.9^{-k})}.$$
Furthermore, $X$ satisfies $X=\Erw X+o(n)$ \whp;
for the number of vertices $v$ with $s(v)=1$ is concentrated by \Lem~\ref{Lemma_Poisson}.
In addition, for all such $v$ with $\sigma(v)=0$ the events that the edge $e_v$ supported by $v$ contains a non-supporting
vertex are mutually independent.
Hence, this number has a binomial distribution and is therefore concentrated by Chernoff bounds
	(\Lem~\ref{Lemma_Chernoff}).
As the same is true of the vertices $v$ with $\sigma(v)=1$, $X$ is concentrated about its expectation.
The other two claims follow from a similar argument.
\qed\end{proof}

Then the above lemma shows that $\abs{U_1}/n\leq\lambda(k-1)\exp(-2\lambda)+O(7^{-k})$ \whp\
Furthermore, the hypergraph $H_{U_1\cup U_2}$ mostly consists of isolated vertices and edges of size $2$ (and only very larger edges).

We now need to analyze how the process for the construction of the set $U$ proceeds.
All vertices in $V\setminus(U_0\cup U_1)$ support at least one edge that does not contain a vertex from $U_0$.
We will now construct sets $U_j$, $j\geq2$, inductively as follows:
	\begin{quote}
	let $U_j$ be the set of all vertices $v\in V\setminus\bigcup_{i<j}U_i$ such that 
	all edges 	supported by $v$ contain a vertex from $\bigcup_{i<j-1}U_i$.
	\end{quote}
Let $U^*=\bigcup_{j\geq2}U_j$.

\begin{lemma}\label{Lemma_deadneighbor}
\Whp\ $H_1$ has the following property.
Let $T$ be a set of size $\leq n/2^{k-2}$.
Then the number $\hat T$ of critical edges that are supported by a vertex  $v\not\in T$ but that contain a vertex from $T$
is bounded by $36k^32^{-k}n$.
\end{lemma}
\begin{proof}
We use a first moment argument.
Let $t=2^{2-k}$ and $\mu=36k^3/2^{k}$.
Then probability of the event described above is bounded by
	\begin{eqnarray*}
	\bink{n}{tn}\bink{m_1}{\mu n}(kt)^{\mu n}&\leq&
		\brk{\frac{\eul}t\bcfr{\eul \lambda kt}{\mu}^{\mu/t}}^{tn}
		\leq\brk{2^k\bcfr{\eul}{9k}^{9k^3}}^{tn}=o(1),
	\end{eqnarray*}
as claimed.
\qed\end{proof}

\begin{lemma}\label{Lemma_branching}
\Whp\ we have $\abs{U^*}\leq n\cdot O(7.2^{-k})$.
\end{lemma}
\begin{proof}
This is based on a branching process argument.
More precisely, we consider the following stochastic process.
At each time, a vertex can be either alive, neutral, or dead.
Initially, all vertices in $U_0$ are dead, all vertices in $U_1$ are alive, and all other vertices are neutral.
In each round of the process an alive vertex $a$ is chosen arbitrarily (once there is no alive vertex left, the process stops).
Every neutral vertex $v$ such that all edges $e$ with
$v\in e$ contain either $a$ or a dead vertex is declared alive, and then $a$ is declared dead.

Let $A_t$ be the set of alive vertices after $t$ steps of the process (in particular, $A_0=U_1$).
Let $T_*=\abs{A_0}$, $T^*=2n\cdot 7.2^{-k}$, and let $T$ be the actual stopping time of the process.
The goal is to show that \whp
	$$T\leq T_*+T^*,$$
which implies that $U^*\setminus U_1\leq T^*$.

To prove this bound, we proceed as follows.
Consider a time $t\leq T_*+T^*$. 
There are several ways in which a neutral vertex $v$ can become alive.
\begin{description}
\item[Case 1: $s(v)=1$.]
	By \Lem~\ref{Lemma_Poisson} the total number of such vertices is bounded by $(1+o(1))\lambda\exp(-\lambda)n$ \whp\
	Moreover, $v$ can become alive only if the unique clause that $v$ supports contains $a$.
	The probability of this event is bounded by $2k/n$.
	Hence, the expected number of new alive vertices that arise in this way is $\leq(1+o(1))2k\lambda\exp(-\lambda)$.
\item[Case 2: $s(v)>1$ and $v$ has a dead neighbor.]
	By \Lem~\ref{Lemma_deadneighbor} and our assumption on $t$, the total number of vertices with a dead neighbor is bounded by $36k^32^{-k}n$.
	If $v$ is declared alive at time $t$, then all edges that contain $v$ but no dead vertex must contain $v$, and there is at least one such edge.
	The probability of this event is bounded by $2k/n$.
	Hence, the expected number of vertices that become alive in this way is $\leq(1+o(1))72k^42^{-k}$.
\item[Case 3: $s(v)>1$ and $v$ does not have a dead neighbor.]
	In this case all $s(v)\geq2$ edges that $v$ supports contain $a$.
	The probability of this event is $O(n^{-2})$.
	Hence, the expected number of vertices that become alive in this way is $o(1)$.
\end{description}
Thus, conditioning on the previous history $\cF_{t-1}$ of the process, we obtain
	$$\Erw\brk{A_{t}-A_{t-1}|\cF_{t-1}}\leq k^5/2^k.$$
Furthermore, for all neutral $v$ the events that $v$ is activated at time $t$ given $\cF_{t-1}$
are mutually independent.
Hence, $A_{t}-A_{t-1}$ given $\cF_{t-1}$ is stochastically dominated by a binomial variable $B_t$ with mean $k^5/2^k$.
Now, if $T\geq T_*+T^*$, then at least $T^*$ vertices got activated by time $T_*+T^*$, i.e.,
	$\sum_{t=1}^{T_*+T^*}B_t\geq T^*.$
Since
	$$\Erw\sum_{t=1}^{T_*+T^*}B_t\leq(T_*+T^*)k^5/2^k<T^*/2,$$
the Chernoff bound from \Lem~\ref{Lemma_Chernoff} shows that $\pr\brk{\sum_{t=1}^{T_*+T^*}B_t\geq T^*}\leq\exp(-\Omega(n))$.
\qed\end{proof}

\noindent\emph{Proof of \Prop~\ref{Prop_ZU}.}
The above discussion allows us to get a close understanding of the combinatorial structure of the hypergraph $H_{U}$.
By \Lem~\ref{Lemma_branching} we have $\abs{U^*}\leq n\cdot O(7.2^{-k})$.
Let $E^*$ be the set of all edges supported by a vertex in $U^*$ that contain a vertex in $U_0\cup U_1$.
\Lem~\ref{Lemma_neighborhood} implies that \whp\ $\abs{E^*}\leq O(k)\abs{U^*}\leq n\cdot O(7.19^{-k})$.
Hence, the set $U_*'$ of all vertices $v\in U_0\cup U_1$ that occur in an edge from $E^*$
has size $|U_*'|\leq n\cdot O(7.18^{-k})$ \whp\
Furthermore, let $U_*$ be the set of all vertices $v\in U_0\cup U_1$ such that either $v\in U_*$
or there is an edge $e$ supported by a vertex in $U_1$ that contains $v$ and another vertex from $U_0\cup U_1$,
or such that $v\in U_0$ occurs in an edge supported by a vertex $w\in U_*'\cap U_1$.
Then by \Lem~\ref{Lemma_U0} we have $\abs{U_*}\leq n O(7.2^{-k})$.

In summary, we have shown that $H_U$ has the following structure \whp
\begin{itemize}
\item The set $U_0$ of non-supporting variables has size $(1+o(1))n\exp(-\lambda)$.
\item There is a set $U_1\setminus U_*$ of size
		$$(1+o(1))n\brk{\lambda(k-1)\exp(-2\lambda)+O(7.17^{-k})}$$
	such that in $H_U$ all vertices in $U_1\setminus U_*$ support exactly one edge
	that contains precisely one other vertex from $U$, which indeed belongs to $U_0$.
\item Apart from these, $H_U$ contains no more than $n O(7.17^{-k})$ further edges.
\end{itemize}
This completes the proof of \Prop~\ref{Prop_ZU}.
\qed

\subsection{Proof of \Prop~\ref{Prop_Z1}}\label{Sec_Z1}

As a first step, we will identify a large set of rigid vertices.
To this end, we need to say something about the number of edges that the vertices in $V\setminus U$ support.
Let $l=10$.

\begin{lemma}\label{Lemma_highdeg}
\Whp\ the number of vertices $v\not\in U$ that support fewer than $l$ edges 
that do not contain a vertex from $U$ is bounded by $\frac{2\lambda^l}{l!}\exp(-\lambda)n$.
\end{lemma}
\begin{proof}
By \Lem~\ref{Lemma_Poisson} the total number of vertices that support fewer than $l$
edges is $\leq\frac{1.01\lambda^l}{l!}\exp(-\lambda)n$ \whp\
Moreover, applying \Lem~\ref{Lemma_deadneighbor} to the set $U$, we see that
no more than $36k^3 n/2^k<0.9\frac{\lambda^l}{l!}\exp(-\lambda)n$ critical edges supported by a vertex in $V\setminus U$ contain a vertex from $U$  \whp\
Each of these edges can create at most one additional vertex in $V\setminus U$ that supports fewer than $l$
edges without a vertex from $U$.
\qed\end{proof}

For each $v\not\in U$ let $s'(v)$ be the number of edges supported by $v$ that do not contain a vertex from $U$.
By the construction of $U$, we have $s'(v)\geq1$ for all $v\not\in U$.
Furthermore, given the sequence $(s'(v))_{v\in V\setminus U}$, the distribution of the sub-hypergraph of $H_1$ induced on
$V\setminus U$ is very simple: it is obtained by choosing, for each vertex $v\in V\setminus U$ independently,
$s'(v)$ edges supported by $v$ and containing a random set of $k-1$ vertices from $V\setminus U$ of color $1-\sigma(v)$.
This follows because the construction of the set $U$ merely imposes the condition that none of the $s'(v)$ remaining edges
supported by $v$ contains a vertex from $U$.

We now decompose the random edges of the sub-hypergraph $H_1-U$ into two portions.
The first portion $\cM$ contains for each vertex $v$ \emph{one} random edge supported by $v$ and containing
$k-1$ vertices of color $1-\sigma(v)$
	(with no vertex from $U$, of course).
The second portion $\cH$ contains the remaining $s'(v)-1\geq0$ random edges supported by $v$ and
containing $k-1$ vertices of color $1-\sigma(v)$ (again, none of them from $U$).
This decomposition will allow us to construct the desired set $R$ in two independent steps.

The first step is in to find a `core' in the hypergraph $\cH$.
\begin{description}
\item[CR1.] Initially, let $S$ contain all $v\in V$ that support at least $l/2$ edges.
\item[CR2.] While there is $v\in S$ that supports $<l/2$ edges consisting of vertices of $S$ only, remove $v$ from $S$.
\end{description}
Let $\cC=S$ be the final outcome of this process.
In order to study $\abs\cC$, we need the following expansion property of the random hypergraph $H_1$.

\begin{lemma}\label{Lemma_coreexp}
\Whp\ the random hypergraph $H_1$ has the following property.
Let $T\subset V$ be a set of size $tn$ with $t\leq1/(\eul^2k\lambda)$.
Then there are no more than $2tn$ edges that are supported by a vertex in $T$
and that contain a second vertex from $T$.
\end{lemma}
\begin{proof}
We use a first moment argument.
The probability that there is a set $T$ that violates the above property is bounded by
	\begin{eqnarray*}
	\bink{n}{tn}\bink{m_1}{2tn}(kt^2)^{2tn}&\leq&
		\brk{\frac\eul t\bcfr{\eul m_1 kt^2}{2tn}^2}^{tn}
		=\brk{\frac\eul t\bcfr{\eul \lambda kt}{2}^2}^{tn}\leq\bcfr t\eul^{tn}=o(1),
	\end{eqnarray*}
as claimed.
\qed\end{proof}

\begin{lemma}\label{Lemma_core}
\Whp\ we have $\abs{V\setminus\cC}\leq\lambda^l\exp(-\lambda)n$.
\end{lemma}
\begin{proof}
Assume that $\abs{V\setminus\cC}>\lambda^l\exp(-\lambda)n$.
By \Lem~\ref{Lemma_highdeg} we may assume that the initial set $S$ contains
at least $n\bc{1-2\lambda^l\exp(-\lambda)/l!}$ vertices.
Hence, if $\abs{V\setminus\cC}>\lambda^l\exp(-\lambda)n$, then at some point the process {\bf CR1--CR2}
must have removed a set $T$ of size $\lambda^l\exp(-\lambda)n/2$ from the original set $S$.
This set $T$ has the property that each vertex in $T$ supports $l/2>2$ edges, each of which must
contain another vertex from $T$.
But by \Lem~\ref{Lemma_coreexp} no such set $T$ exists \whp
\qed\end{proof}

Having constructed the set $\cC$, we are now going to `attach' more vertices from $V\setminus U$ to it via
the following process.
\begin{description}
\item[A1.] Let $\cA_0=\cC$.
\item[A2.] For $t\geq1$, let $\cA_t$ be the set of all vertices $v\in V\setminus U$ such that either $v\in\cA_{t-1}$
		or the edge $e\in\cM$ supported by $v$ has its other $k-1$ vertices in $\cA_{t-1}$.
\end{description}

Let $\cA=\bigcup_{t=0}^\infty\cA_t$.
Observe that actually $\cA=\cA_n$, i.e., the process becomes stationary after at most $n$ steps.

\begin{lemma}\label{Lemma_attach}
\Whp\ the outcome of the above process satisfies $|\cA|=|V\setminus U|-o(n)$.
\end{lemma}
\begin{proof}
Let $\cA_t$ be the set constructed after $t$ steps of the above process, with $\cA_0=\cC$ and $\cA_{-1}=\emptyset$.
Let $\cH_t$ be the history of the process up to time $t$.
Let $v\in V\setminus(U\cup\cA_t)$ be a vertex, and let $e_v\in\cM$ be the random edge supported by $v$.
The only conditioning that $\cH_t$ imposes on $e_v$ is that $e_v$ has at least one vertex $w\neq v$
that does not lie in $\cA_{t-1}$.
Hence,
	\begin{equation}\label{eqattach1}
	\pr\brk{v\not\in\cA_{t+1}|\cH_t}
		=\pr\brk{e_v\setminus\cbc v\not\subset\cA_{t}|\cH_t}\leq(k-1)\cdot\frac{\abs{V\setminus(U\cup\cA_t)}}{\abs{V\setminus(U\cup\cA_{t-1})}}.
	\end{equation}

To analyze the quantity on the right, let $a_t=\abs{\cA_t}/|V\setminus U|$ for $t\geq-1$.
Then \Lem~\ref{Lemma_core} implies that \whp\ $a_0\geq1-\lambda^l\exp(-\lambda)$.
With this notation, (\ref{eqattach1}) reads
	$$\Erw\brk{1-a_{t+1}|\cH_t}\leq\frac{(k-1)(1-a_t)^2}{1-a_{t-1}}.$$
Furthermore, given $\cH_t$, for all vertices $v\in V\setminus(U\cup\cA_t)$ the events $\cbc{v\not\in\cA_{t+1}}$
are mutually independent (because each is determined by the edge $e_v$ supported by $v$).
Therefore, the number of $v\in V\setminus(U\cup\cA_t)$ such that $v\not\in\cA_{t+1}$ is stochastically dominated
by a binomial distribution with mean $|V\setminus U|\cdot \frac{(k-1)(1-a_t)^2}{1-a_{t-1}}.$
By Chernoff bounds, with probability $1-o(1/n)$ we therefore see that the number of $v\in V\setminus(U\cup\cA_t)$ such that $v\not\in\cA_{t+1}$
is $|V\setminus U|\cdot \frac{(k-1)(1-a_t)^2}{1-a_{t-1}}+o(n)$.
Hence,
	\begin{equation}\label{eqattach2}
	\pr\brk{a_{t+1}<1-\frac{(k-1)(1-a_t)^2}{1-a_{t-1}}+o(1)|\cH_t}=o(1/n),
	\end{equation}
and thus the above holds for all $t\geq1$ \whp\

Now, consider the (deterministic) recurrence
	$$\alpha_0=\lambda^l\exp(-\lambda),\quad\alpha_{t+1}=1-\frac{(k-1)(1-\alpha_t)^2}{1-\alpha_{t-1}}.$$
It is straightforward to verify that $\lim_{t\ra\infty}\alpha_t=1$.
Therefore, (\ref{eqattach2}) implies that \whp\
	$$\lim_{t\ra\infty}\abs{V\setminus(U\cup\cA_t)/n}=0,$$
and thus $\abs{V\setminus(U\cup\cA)}=o(n)$ \whp
\qed\end{proof}

\medskip\noindent
\emph{Proof of \Prop~\ref{Prop_Z1}.}
We are left to show that \whp\ all vertices in $\cA$ are $n/k^3$-rigid.
We start by proving that \whp\ all vertices in $\cC$ are $n/k^3$-rigid.
Suppose that there is another $2$-coloring $\tau$ of $\cC$ such that the set
	$$\Delta=\cbc{v\in\cC:\sigma(v)\neq\tau(v)}$$
has size $0<\abs{\Delta}<n/k^3.$
By the construction of $\cC$, each vertex $v\in\Delta$ supports at least 3 edges
that consist of vertices in $\cC$ only.
As these edges are bichromatic under $\tau$, each of them must contain a second vertex in $\Delta$.
Hence, there are at least $3\abs{\Delta}$ edges that are supported by a vertex in $\Delta$ (under $\sigma$)
and that contain a second vertex in $\Delta$.
But \Lem~\ref{Lemma_coreexp} shows that \whp\ there is no such set $\Delta$ of size $0<\abs{\Delta}<n/k^3.$
This shows that all vertices $\cC$ are $n/k^3$-rigid \whp\

Furthermore, the construction of $\cA$ ensures that any $2$-coloring $\tau$ of $H_1$
such that $\tau(v)\neq\sigma(v)$ for some $v\in\cA$ is indeed such that
$\tau(w)\neq\sigma(w)$ for some $w\in\cC$.
This shows that any $v\in\cA$ is $n/k^3$-rigid \whp, because any $w\in\cC$ is.
\qed

\section{A closer look at the internal entropy: proof of \Cor~\ref{Cor_cond}}

\subsection{Outline}

Throughout this section, we let $f_0(n)$ denote a function such that $f_0(n)=o(n)$ as $n\ra\infty$.
Let $\sigma\in\cbc{0,1}^n$ be such that $||\sigma^{-1}(0)|-|\sigma^{-1}(1)||\leq f_0(n)$.
In addition, let $\sigma_0$ be an equitable $2$-coloring.
To prove \Cor~\ref{Cor_cond} we need to prove that the size $\abs{\cC(\sigma)}$
of the local cluster in the planted model $H_k(n,m,\sigma)$ is tightly concentrated.
To accomplish that, we need to study the set $U$ from \Sec~\ref{Sec_localCluster}.
That is, $U\subset V$ is constructed as follows.
\begin{enumerate}
\item Initially, let $U$ consist of all vertices that do not support any edges.
\item While there is a vertex $v\not\in U$ that does not support an edge that does not contain a vertex from $U$, add $v$ to $U$.
\end{enumerate}
As a first step, we are going to show that $|U|$ is tightly concentrated.
More precisely, in \Sec~\ref{Sec_Uconc} we will prove the following.

\begin{proposition}\label{Prop_Uconc}
For any two functions $f_0(n)=o(n)$, $f_1(n)=o(n)$ there is a function $f_2(n)=o(n)$ such that
	$$\pr\brk{\abs{|U|-\Erw_{H_k(n,m,\sigma_0)}|U|}>f_2(n)}\leq\exp(-f_1(n)).$$
\end{proposition}

We also need the following simple expansion properties.

\begin{lemma}\label{Lemma_propertyX}
\Whp\ both $H_k(n,m)$ and $H_k(n,m,\sigma)$ have the following property.
\begin{equation}\label{eqpropertyX}
\parbox[t]{14cm}{For any set $S\subset V$ of size $|S|\leq2^{-k^2}n$ the number of edges $e$ 
that contain at least two vertices from $S$ is bounded by $1.01|S|$.}
\end{equation}
Furthermore, with probability $1-\exp(-\Omega(n))$,  $H_k(n,m,\sigma)$ has the following property.
\begin{equation}\label{eqpropertyXX}
\parbox[t]{14cm}{For any set $S\subset V$ of size $2^{-k^2}n<|S|\leq n/k^3$ the number of critical edges $e$
that contain at least two vertices from $S$ is bounded by $1.01|S|$.}
\end{equation}
\end{lemma}
\begin{proof}
This follows from a simple first moment argument similar to the one in the proof of \Lem~\ref{Lemma_coreexp}.
\qed\end{proof}

Using \Prop~\ref{Prop_Uconc} and \Lem~\ref{Lemma_propertyX}, we will derive the following
in \Sec~\ref{Sec_Ucomp}.

\begin{proposition}\label{Prop_Ucomp}
Let
	$\nu_k(n,m)=\Erw_{H_k(n,m,\sigma_0)}\ln\cC(\sigma_0).$
For any $f_0(n),f_1(n)=o(n)$ there is a function $f_3(n)=o(n)$ such that
	\begin{equation}\label{eqPropUcomp1}
	\pr_{H_k(n,m,\sigma)}\brk{\abs{\nu_k(n,m)-\ln\cC(\sigma)}>f_3(n)\mbox{ and (\ref{eqpropertyX}) holds}}\leq\exp(-f_1(n)).
	\end{equation}
Furthermore, for any $1\leq j\leq n^{2/3}$ we have
	\begin{equation}\label{eqPropUcomp2}
	0\leq \nu_k(n,m)-\nu_k(n,m+j)=o(n^{3/4}).
	\end{equation}
\end{proposition}

\noindent\emph{Proof of \Cor~\ref{Cor_cond}.}
To begin, let us fix a small $\eps>0$.
Our first goal is to show that there exists a density $r=r(n)$ such that
	\begin{equation}\label{eqCorcond1}
	\frac1n\Erw_{H_k(n,m,\sigma)}\ln|\cC(\sigma)|\sim\frac1n\ln\Erw_{H_k(n,m)} Z-\eps.
	\end{equation}

To prove~(\ref{eqCorcond1}), it is easier to work with the random hypergraph $H_k(n,p,\sigma)$
in which each $e\subset V$ of size $k$ that is bicolored under $\sigma$ is inserted with probability $p$ independently.
Then for any \emph{fixed} $n$, the function
	$$F_n(p)=\frac1n\Erw_{\sigma,H_k(n,p,\sigma)}\ln|\cC(\sigma)|$$
is a polynomial in $p$.
Furthermore, it is clear that $F_n(p)\ra\ln2$ as $p\ra0$, and $F_n(p)\ra o(1)$ as $p\ra1$.
For any $p$ we let $\rho(p)\geq 0$ be such that the expected number of edges in $H_k(n,p,\sigma)$ equals $\rho(p)n$.
Then by the mean value theorem, there exists $p$ such that $F_n(p)\sim\frac1n\ln\Erw_{H_k(n,\lceil \rho(p)n\rceil)} Z-\eps$.
Since the acutal number of edges of $H_k(n,p,\sigma)$ is binomially distributed and therefore tightly concentrated about $\rho(p)n$,
the `continuity property' (\ref{eqPropUcomp2}) ensures that
	$$\frac1n\Erw_{H_k(n,\lceil \rho(p)n\rceil,\sigma)}\ln|\cC(\sigma)|\sim F_n(p)\sim\frac1n\ln\Erw_{H_k(n,\lceil \rho(p)n\rceil)} Z-\eps.$$
Setting $r_\eps(n)=\rho(p(n))$, we obtain~(\ref{eqCorcond1}).

For this density $r=r_\eps(n)$ there exists a function $f_1(n)=o(n)$  such that
	\begin{equation}\label{eqplantingX}
	\ln(g_{k,n,m}\brk\cB)\leq\ln(p_{k,n,m}\brk\cB)+f_1(n)
		\qquad\mbox{for any event }\cB\neq\emptyset.
	\end{equation}
Let $f_0(n)$ be such that with probability $1-\exp(-2f_1(n))$, a random $\sigma\in\cbc{0,1}^n$ satisfies
$||\sigma^{-1}(0)|-|\sigma^{-1}(1)||\leq f_0(n)$.
Combining \Lem~\ref{Lemma_propertyX}, \Prop~\ref{Prop_Ucomp},~(\ref{eqCorcond1})  and (\ref{eqplantingX}),
we see that for these densities $r_\eps(n)$, \whp\ a random pair $(H,\sigma)$ chosen from the
Gibbs distribution is such that
	\begin{equation}\label{eqCorcond2}
	\frac1n\ln\abs{\cC(\sigma)}\geq\frac1n\ln\Erw\brk Z-\eps\geq \frac1n\ln Z(H)-2\eps.
	\end{equation}
Since~(\ref{eqCorcond2}) holds \whp\ for any fixed $\eps>0$, there \emph{exist} sequences $\eps(n)\ra0$, $r(n)$
as desired.
\qed

\subsection{Proof of \Prop~\ref{Prop_Uconc}}\label{Sec_Uconc}

We are going to trace the process for the construction of the set $U$ via the method of differential equations~\cite{Wormald}.
To obtain sufficient concentration from this approach, we will have to modify the process slightly.
The modified process will yield a \emph{subset} $U_*\subset U$, whose size is tightly concentrated.
We will then see how $U_*$ can be enhanced to a superset $U^*\supset U$, whose size
does not exceed the size of $U_*$ significantly with a very high probability.

Our construction of $U_*$ comes with a parameter $\omega\geq\omega_0$, where $\omega_0$ denotes a large constant
	(later we will let $\omega\ra\infty$ slowly as $n\ra\infty$).
To construct $U_*$, we consider a similar process as in the proof of \Lem~\ref{Lemma_branching},
but we only run this process on the set $V'$ of vertices that support at most $\omega$ clauses.
In each step, any vertex $w\in V'$ is either alive, dead, or neutral.
Initially, all vertices in $V'$ that do not support a clause are alive, and all others are neutral.
The process stops once there is no alive vertex left.
In each step, an alive vertex $v$ is chosen randomly.
Let $d_v$ be the number edges $e_1,\ldots,e_{d_v}$ supported by neutral vertices in which $v$ occurs.
\begin{description}
\item[Case 1: $d_v\leq\omega$.] All of $e_1,\ldots,e_{d_v}$ are deleted from the hypergraph.
\item[Case 2: $d_v>\omega.$]
	In this case $\omega$ edges amongst $e_1,\ldots,e_{d_v}$ are chosen randomly
	and are deleted from the hypergraph.
	Moreover, the remaining $d_v-\omega$ edges are \emph{changed} as follows.
	Suppose that the deleted edges are $e_1,\ldots,e_{\omega}$.
	Then $v$ is replaced in each edge $e\in\cbc{e_{\omega+1},\ldots,e_{d_v}}$ independently
	by a random vertex $w\neq v$ with $\sigma(w)=\sigma(v)$ that is not dead and that does not belong to $e$ already;
	if there is no such vertex $w$ left, the process stops.
\end{description}
Finally, all neutral vertices that do not support an edge anymore (after the edge deletions described above)
are declared alive, and $v$ is declared dead.
Let $T$ be the stopping time of the process, and let $U_*$ be the set of dead vertices upon termination.
Then $|U_*|=T$.

The difference between the above process and the actual construction of $U$ is that the latter runs on the \emph{entire}
set $V$ (not just $V'$) and that it \emph{always} removes the $e_1,\ldots,e_{d_v}$.
Therefore, $U_*\subset U$.

To trace the construction of $U_*$, we need to define a few random variables.
For each $1\leq s\leq\omega$ and each $1\leq l \leq s$ let
$X_t(s,l)$ denote the number of neutral vertices that support $s$ vertices in total,
out of which $l$ do not contain a vertex that has died by the end of step $t$.
In addition, let $A_t$ signify the number of alive vertices.
Let $(\cF_t)_{t\geq0}$ be the filtration generated by the random variables
	$X_t(s,l)$ and $A_t$.

Let $\cD_s$ be the number of vertices that support precisely $s$ edges $(s\geq0)$.
Moreover, let $\cD_{>\omega}$ be the number of vertices that support more than $\omega$ edges.

\begin{lemma}\label{Lemma_omegadeg}
We have
	$$\pr\brk{\cD_{>\omega}>\exp(-\omega)n}\leq\exp\brk{-\frac{n}{2\exp(2\omega)r}}.$$
Furthermore, for any $0\leq s\leq\omega$ we have
	$$\pr\brk{|\cD_s-\Erw\cD_s|>\exp(-\omega^2)n}\leq\exp\brk{-\frac{n}{2\exp(2\omega^2)r}}.$$
\end{lemma}
\begin{proof}
For each vertex $v$ the number $s(v)$ of edges supported by $v$ has a binomial distribution with mean $\lambda=kr/(2^{k-1}-1)$.
Assuming that ($k$ and thus) $\lambda$ is sufficiently large, and choosing $\omega_0$ big enough, we see from Chernoff bounds that
	$\pr\brk{s(v)>\omega}\leq\exp(-8\omega)$.
Hence, $\Erw\cD_{>\omega}\leq n\exp(-8\omega)$.
Furthermore, $\cD_{>\omega}$ satisfies a Lipschitz condition: adding or removing a single edge can alter the value of $\cD_{>\omega}$ by at most one.
Therefore, the first assertion follows from Azuma's inequality.
Similarly, adding or removing a single edge can change the value of $\cD_s$ by at most one, and
thus Azuma's inequality also implies the second claim.
\qed\end{proof}

\begin{lemma}\label{Lemma_diffeq}
For any $1\leq t<\min\cbc{T,n/k^2}$ we have
	\begin{eqnarray}\label{eqDelta1}
	\Erw\brk{X_{t+1}(s,l)|\cF_t}&=&
					X_t(s,l)\bc{1-\frac{l(k-1)}{n-t}}+
							X_t(s,l+1)\cdot\frac{(l+1)(k-1)}{n-t}
							+o_\omega(1),\\
	\Erw\brk{A_{t+1}|\cF_t}&=
			\frac{k-1}{n-t}\sum_{s=1}^\omega X_t(s,1)+o_\omega(1)
			\label{eqDelta2}
	\end{eqnarray}
Furthermore,
	\begin{equation}\label{eqDelta3}
	|X_{t+1}(s,\lambda)-X_t(s,\lambda)|\leq\omega,\quad|A_{t+1}-A_t|\leq\omega
	\end{equation}
with certainty.
\end{lemma}
\begin{proof}
This is a standard argument for a differential equations analysis, based on
the following observation (`method of deferred decisions'):
	given the history $\cF_t$ of the process up to time $t$, for each neutral vertex $w$ each
	remaining edge $e$ supported by $w$ is conditioned \emph{only} to the effect that
	$e$ does not contain a vertex that has died by time $t$.
Thus, the alive vertex $v$ chosen at time $t+1$ has a probability of $1-(1-1/(n-t))^{k-1}\sim (k-1)/(n-t)$ of occurring
in each edge supported by a neutral vertex, and these events are independent for all such edges.
Furthermore, since each neutral vertex only supports $\leq\omega$ edges (as we confine ourselves to the set $V'$),
the probability that $v$ occurs in two such edges is $o(1)$.

This means that given $\cF_t$ the expected number of vertices that support $s$ edges in total, out of which $l$ are left after time $t$,
and which support an edge in which $v$ occurs, equals
	\begin{equation}\label{eqDelta11}
	X_t(s,l)\frac{l(k-1)}{n-t}+o(1).
	\end{equation}
Furthermore, the given $\cF_t$ expected number of vertices that support $s$ edges in total with $l+1$ left after time $t$
amongst which precisely one contains $v$, is
	\begin{equation}\label{eqDelta12}
	X_t(s,l+1)\cdot\frac{(l+1)(k-1)}{n-t}+o(1).
	\end{equation}

To obtain~(\ref{eqDelta1}) from this, we need to take into account the `exceptional' case~2 of the process.
But since the expected number of occurrences of $v$ given $\cF_t$ is bounded by
$\lambda n/(n-t)\leq2\lambda$, and since this number is binomially distributed, the probability that $v$
occurs in more than $\omega$ edges supported by neutral vertices is bounded by $\exp(-\omega)$.
This estimate in combination with~(\ref{eqDelta11}) and~(\ref{eqDelta12}) yields~(\ref{eqDelta1}).
Equation~(\ref{eqDelta2}) follows from a similar argument, and (\ref{eqDelta3}) is immediate from the construction.
\qed\end{proof}

\begin{corollary}\label{Cor_diffeq}
There exists a number $0<\mu=\mu(k,r)\leq2n\exp(-\lambda)$ and a function $\delta_\omega=o_\omega(1)$ such that
	\begin{equation}\label{eqdiffeq}
	\pr\brk{|T/n-\mu|\leq \delta_\omega}\geq1-\exp\brk{-\frac{n}{\exp(\omega^3)}}.
	\end{equation}
\end{corollary}
\begin{proof}
\Lem~\ref{Lemma_omegadeg} and \Lem~\ref{Lemma_diffeq} verify the assumptions of~\cite[\Thm~5.1]{Wormald}
for times $t\leq n/k^2$.
Furthermore, \Prop~\ref{Prop_ZU} shows that $T\leq 2n\exp(-\lambda)<n/k^2$ \whp\
Therefore, we can apply~\cite[\Thm~5.1]{Wormald} to obtain~(\ref{eqdiffeq}).
\qed\end{proof}

\begin{remark}
The `method of differential equations'~\cite[\Thm~5.1]{Wormald} actually shows that the random variables
$X_t(s,l)$ closely trace a system of ordinary differential equations.
From these the number $\mu$ in \Cor~\ref{Cor_diffeq} could, in principle, be worked out precisely
for any given $k,r$.
However, for our purposes it is not important to know $\mu$ precisely.
In the proof of \Cor~\ref{Cor_diffeq} it is important to use the differential equations approach as in~\cite{Wormald}
to ensure sufficient concentration.
\end{remark}

Since $\abs{U_*}=T$ and $U_*\subset U$, \Cor~\ref{Cor_diffeq} provides a lower bound on the size of $U$.
As a next step, we will derive an (asymptotically) matching upper bound.

\begin{lemma}\label{Lemma_Uupper}
There is a function $\delta_\omega=o_\omega(1)$ such that
	$$\pr\brk{|U\setminus U_*|>\delta_\omega n}\leq3\exp\brk{-\frac{n}{\exp(\omega^3)}}.$$
\end{lemma}
\begin{proof}
We use a similar argument as in the proof of \Lem~\ref{Lemma_branching}.
Namely, having constructed $U_*$, we commence a second process.
Again, in the course of this process vertices can be alive, dead, or neutral.
Initially, all vertices in $U_*$ are dead.
Furthermore, a vertex $v\not\in U_*$ is declared alive if either
$v\in V\setminus V'$ (i.e., $v$ supports more than $\omega$ edges),
or $v$ supports an edge that contains a vertex that occurs in more than $\omega$ edges supported by other vertices.
All other vertices are neutral.
From this initial state, the process proceeds just like the construction of the set $U$.
Namely, in each step an alive vertex $v$ is chosen, unless there is none left, in which case the process stops.
Then, all vertices $w$ such that each edge $e$ supported by $w$ contains either $v$ or a dead vertex are declared alive,
and $v$ dies.
Clearly, the set of dead vertices of this process upon termination contains $U$.

By \Cor~\ref{Cor_diffeq} we may assume that $|U_*|\leq2n\exp(-\lambda)$.
Standard arguments (similar to the proof of \Lem~\ref{Lemma_deadneighbor}) show that with probability
$\geq1-\exp\brk{-\frac{n}{\exp(\omega^3)}}$ the total number of vertices that are alive initially is $o_\omega(1)n$.
Furthermore, using stochastic dominance as in the proof of \Lem~\ref{Lemma_branching}, one can show
that the above process will terminate after only $o_\omega(1)n$ steps with probability
	$\geq1-\exp\brk{-\frac{n}{\exp(\omega^3)}}$.
\qed\end{proof}

\noindent\emph{Proof of \Prop~\ref{Prop_Uconc}.}
\Cor~\ref{Cor_diffeq} and \Lem~\ref{Lemma_Uupper} show that there exist $\mu>0$ and
for any $\omega\geq\omega_0$ some $\delta_\omega>0$ such that
	\begin{equation}\label{eqUconcFinal}
	\pr\brk{||U|-\mu n|>\delta_\omega n}\leq \exp(-n/\exp(\omega^4)),
	\end{equation}
where $\delta_\omega\ra0$ as $\omega\ra\infty$.
We may assume that the given function $f_1(n)=o(n)$ satisfies $f_1(n)\geq\sqrt n$, and that $\delta_\omega\geq1/\omega$.
Given such a $f_1(n)$, we can choose a slowly growing function $\omega=\omega(n)\leq\ln\ln n$ such that
$\exp(-n/\exp(\omega^4))\leq\exp(-f_1(n))$.
Then~(\ref{eqUconcFinal}) implies that $|\Erw|U|-\mu n|\leq2\delta_{\omega(n)}$.
Thus, setting $f_2(n)=2\delta_{\omega(n)}$ and invoking~(\ref{eqUconcFinal}) once more completes the proof.
\qed

\subsection{Proof of \Prop~\ref{Prop_Ucomp}}\label{Sec_Ucomp}

To prove \Prop~\ref{Prop_Ucomp}, it will be easier to work with a slightly different distribution over the
hypergraphs that for which $\sigma$ is a $2$-coloring.
Namely, $H_k(n,p,\sigma)$ denote a random hypergraph obtained by including each possible edge that is 2-colored under
$\sigma$ with probability $p$ independently.
\emph{Throughout this section, we fix $p$ so that the expected number of edges of $H_k(n,p,\sigma)$ is equal to $m$.}
Due to our assumptions on $\sigma$, this means that $p\sim m/((1-2^{1-k})\bink nk)$.

\begin{lemma}\label{Lemma_Hnp}
For any event $E$,
	$\pr\brk{H_k(n,m,\sigma)\in E}\leq O(\sqrt n)\pr\brk{H_k(n,p,\sigma)\in E}.$
\end{lemma}
\begin{proof}
In $H_k(n,p)$ the total number of edges has a binomial distribution with mean $m$.
Therefore, the probability that $H_k(n,p)$ has exactly $m$ edges is $\Omega(1/\sqrt m)=\Omega(1/\sqrt n)$.
Furthermore, given that its total number of edges is $m$, $H_k(n,p)$ is uniformly distributed over all such
hypergraphs for which $\sigma$ is a $2$-coloring.
\qed\end{proof}

The argument for the proof of \Prop~\ref{Prop_Ucomp} basically is as follows.
We will see that (essentially) all vertices in $V\setminus U$ are rigid, and thus the entropy of the local cluster
stems solely from variations of the colors in $U$.
Furthermore, the hypergraph induced on $U$ by the edges
do not already contain two vertices from $V\setminus U$ with different colors is sub-critical,
i.e., it decomposes into small (at most $\ln n$ but mostly constant-sized) connected components.
Now, for each `type' (i.e., isomorphism class) of component the number of occurrences of this type
is tightly concentrated (similarly as in a subcritical random graph).
This implies concentration of the total number of colorings on $V\setminus U$ because
the total number is simply the sum of the numbers of colorings of the components.
Let us now carry out the details.

Consider the following way to construct a set $\cC\subset V$ of vertices of $H_k(n,m,\sigma)$ (cf.\ \Sec~\ref{Sec_Z1}.).
Let $l=10$.
\begin{description}
\item[C1.] Initially, let $\cC$ contain all $v\in V$ that support at least $l/2$ edges.
\item[C2.] While there is $v\in\cC$ that supports $<l/2$ edges consisting of vertices of $\cC$ only, remove $v$ from $\cC$.
\item[C3.] While there is a vertex $v\in V\setminus\cC$ that supports an edge $e$ such that $e\setminus\cbc v\subset\cC$,
		add $v$ to $\cC$.
\end{description}

\begin{proposition}\label{Prop_cC}
For any function $g_1(n)=o(n)$ there is a function $g_2(n)=o(n)$ such that
with probability $\geq1-\exp(-g_1(n))$ the set $\cC$ has the following properties.
\begin{enumerate}
\item $|V\setminus\cC|=|U|+g_2(n)$.
\item Either~(\ref{eqpropertyX}) is violated, or
		any $2$-coloring $\tau\in\cC(\sigma)$ satisfies
			$\tau(v)=\sigma(v)$ for all $v\in\cC$.
\end{enumerate}
\end{proposition}
\begin{proof}
The same arguments as in \Sec~\ref{Sec_Z1} apply.
\qed\end{proof}

For a set $C\subset V$ let $\cA(C)$ be the set of all $e\subset V$, $|e|=k$,
that have neither of the following two properties.
\begin{enumerate}
\item $e\subset C$.
\item There is a color $i\in\cbc{0,1}$ such that $|e\cap\sigma^{-1}(i)|=1$ and $|e\cap\sigma^{-1}(1-i)\cap C|=k-1$.
		(In other words, $e$ is critical with respect to $\sigma$ and has $k-1$ vertices, not including the supporting one, in $C$.)
\end{enumerate}

The reason why it is easier, for the present context, to work with the $H_k(n,p,\sigma)$ model is the following simple observation.
If we condition on the outcome $\cC\subset V$ of the process {\bf C1}--{\bf C3}, each $e\in\cA(\cC)$ is present as an edge
in $H_k(n,p,\sigma)$ with probability $p$ independently.
That is, the distribution of $H_k(n,p,\sigma)$ \emph{outside} the `core' $\cC$ can be captured very easily.

Given the outcome $\cC$ of the process {\bf C1}--{\bf C3}, let
$\bar H_k(n,p,\sigma,\cC)$ denote the random hypergraph on $V\setminus\cC$ in which we include the set $e\setminus\cC$
for each edge $e$ of $H_k(n,m,\sigma)$ such that $|e\setminus\cC|\geq2$.

\begin{lemma}\label{Lemma_nogiant}
Suppose $|V\setminus\cC|\leq3\exp(-\lambda)n$.
\Whp\ all connected components of $\bar H_k(n,p,\sigma,\cC)$ have size $O(\ln n)$.
Furthermore, for any $\omega=\omega(n)\ra\infty$ 
the expected number of vertices of $\bar H_k(n,p,\sigma,\cC)$ that belong to components of size at least $\omega$ is bounded by
$\exp(-\Omega(\omega)) n$.
\end{lemma}
\begin{proof}
Given the above observation, the assertion is a direct consequence of the result on the `giant component' phase
transition in random non-uniform hypergraphs from~\cite{SPS85}.
\qed\end{proof}

Let $\cT$ be the set of all equivalence classes with respect to isomorphism of hypergraphs with edges of size $\leq k$.
An \emph{isolated copy} of $T\in\cT$ in $\bar H_k(n,p,\sigma,\cC)$ is a subset $S\subset V\setminus\cC$ such
that $S$ is a component of $\bar H_k(n,p,\sigma,\cC)$ and such that the sub-hypergraph induced on $S$ is isomorphic to $T$.
Let $Y_{T,\cC}$ signify the number of isolated copies of $T\in\cT$ in $\bar H_k(n,p,\sigma,\cC)$ (given the set $\cC$).

\begin{lemma}\label{Lemma_compconc}
For any $T\in\cT$ and any $d>0$ we have
	\begin{equation}\label{eqcompconc}
	\pr\brk{|Y_{T,\cC}-\Erw Y_{T,\cC}|>d}\leq\exp\bc{-\frac{d^2}{16k^2m}}.
	\end{equation}
Furthermore, 
if $|V\setminus\cC|\leq3\exp(-\lambda)n$, then
for any $\omega=\omega(n)\ra\infty$ we have
	$$\sum_{T\in\cT:|V(T)|\leq\omega}|V(T)|\cdot\Erw\brk{Y_{T,\cC}}\geq(1-\exp(-\Omega(\omega)))n.$$
\end{lemma}
\begin{proof}
The random variable $Y_T$ satisfies a Lipschitz condition:
	either adding or removing an edge to/from $\bar H_k(n,p,\sigma,\cC)$ can change
		$Y_T$ by at most $k$.
Therefore, the first assertion follows from Azuma's inequality.
The second one is an immediate consequence of~(\ref{eqcompconc}) 
and \Lem~\ref{Lemma_nogiant}.
\qed\end{proof}

We will now drop the conditioning upon the outcome of the process {\bf C1}--{\bf C3}.
That is, we let $\bar H_k(n,p,\sigma)$ be the random hypergraph obtained by first constructing $\cC$ in $H_k(n,p,\sigma)$
and then performing the construction of $H_k(n,p,\sigma,\cC)$.
For each $T\in\cT$ let $Y_T$ be the number of isolated copies of $T$ in $\bar H_k(n,p,\sigma)$.

\begin{corollary}\label{Cor_compconc}
For any function $g_1(n)=o(n)$ and any $\omega=\omega(n)\ra\infty$ there exists $g_2(n)=o(n)$ such that the following is true.
For each $T\in\cT$ there is a number $y_T=y_T(k,r)\geq0$ such that with probability $\geq1-\exp(-g_1(n))$ 
either~(\ref{eqpropertyX}) is violated or the following is true.
\begin{enumerate}
\item All but $g_2(n)$ vertices of $\bar H_k(n,p,\sigma)$ belong to a component on $\leq\omega$ vertices.
\item We have
		$\sum_{T\in\cT}|V(T)|\cdot|Y_{T}-y_Tn|\leq 2g_2(n)$.
\end{enumerate}
\end{corollary}
\begin{proof}
This is immediate from \Prop~\ref{Prop_cC} (which, crucially, shows that $|\cC|$ is tightly concentrated) and \Lem~\ref{Lemma_compconc}.
\qed\end{proof}

For each $T\in\cT$ let $z_T$ denote the number of $2$-colorings of $T$.
Furthermore, let $Z(\bar H_k(n,p,\sigma))$ denote the number of $2$-colorings of $\bar H_k(n,p,\sigma)$.

\begin{corollary}\label{Cor_zcomp}
For any function $g_1(n)=o(n)$ there exists $g_2(n)=o(1)$ such that
with probability $\geq1-\exp(-g_1(n))$ either~(\ref{eqpropertyX}) is violated or we have
	$$\abs{\frac1n\ln Z(\bar H_k(n,p,\sigma))-\sum_{T\in\cT}y_Tz_T}\leq g_2(n).$$
\end{corollary}

\noindent\emph{Proof of \Prop~\ref{Prop_Ucomp}.}
Let us first deal with the random hypergraph $H_k(n,p,\sigma)$.
Suppose that $|\cC|\geq n(1-2\exp(-\lambda))$.
Then any $2$-coloring $\tau$ of $\bar H_k(n,p,\sigma,\cC)$ yields an element
of the local cluster $\cC_{H_k(n,p,\sigma,\cC)}(\sigma)$ of $H_k(n,p,\sigma,\cC)$ by letting $\tau(v)=\sigma(v)$ for all $v\in\cC$.
Therefore, \Prop~\ref{Prop_cC} and \Cor~\ref{Cor_zcomp} imply that
	\begin{equation}\label{eqUcomp1}
	\pr_{H_k(n,p,\sigma)}\brk{\frac1n\ln|\cC(\sigma)|\geq\sum_{T\in\cT}y_Tz_T-o(1)}\geq1-\exp(-f_1(n)).
	\end{equation}

Conversely, consider a coloring $\tau$ of $H_k(n,p,\sigma)$.
By \Prop~\ref{Prop_cC}, 
either~(\ref{eqpropertyX}) is violated or $\tau(v)=\sigma(v)$ for all $v\in\cC$.
Assume that the latter is true.
Then $\tau$ induces a $2$-coloring of $\bar H_k(n,p,\sigma,\cC)$.
	\begin{equation}\label{eqUcomp2}
	\frac1n\ln\cC_{H_k(n,p,\sigma)}(\sigma)\leq\sum_{T\in\cT}y_Tz_T+o(1).
	\end{equation}
Hence,
	\begin{equation}\label{eqUcomp3}
	\pr_{H_k(n,p,\sigma)}\brk{\mbox{either~(\ref{eqpropertyX}) is violated or }\frac1n\ln|\cC(\sigma)|\leq\sum_{T\in\cT}y_Tz_T+o(1)}\geq1-\exp(-f_1(n)).
	\end{equation}
Combining~(\ref{eqUcomp1}) and~(\ref{eqUcomp3}) with \Lem~\ref{Lemma_Hnp}, we obtain~(\ref{eqPropUcomp1}).

Finally, to obtain~(\ref{eqPropUcomp2}), observe that adding a further of $n^{2/3}$ edges to $H_k(n,m,\sigma)$  will simply 
can just connect at most $n^{2/3}$ components of $\bar H_k(n,p,\sigma,\cC)$ with the set $\cC$.
\Lem~\ref{Lemma_nogiant} shows that \whp\ all of these components have size $\leq n^{0.01}$.
Hence, \whp\ the total reduction in the number of $2$-colorings is $\leq n^{2/3+0.01}=o(n^{3/4})$.
\qed


\begin{thebibliography}{99}

\bibitem{Barriers}
D.~Achlioptas, A.~Coja-Oghlan: 
Algorithmic barriers from phase transitions.
Proc.~49th FOCS (2008) 793--802.

\bibitem{AKKT}
D.~Achlioptas, J.H.~Kim, M.~Krivelevich, P.~Tetali:
Two-coloring random hypergraphs.
Random Structures and Algorithms {\bf 18} (2002), 249--259.

\bibitem{nae}
D.~Achlioptas, C.~Moore:
Random $k$-SAT: two moments suffice to cross a sharp threshold.
SIAM Journal on Computing {\bf 36} (2006) 740--762.

\bibitem{AchNaor}
D.~Achlioptas, A.~Naor:
The two possible values of the chromatic number of a random graph.
Annals of Mathematics {\bf 162} (2005), 1333--1349.

\bibitem{yuval}
D.~Achlioptas, Y.~Peres:
The threshold for random $k$-SAT is $2^k \ln 2 - O(k)$.
Journal of the AMS \textbf{17} (2004) 947--973.

\bibitem{fede}
D.~Achlioptas, F.~Ricci-Tersenghi:
On the solution space geometry of random constraint satisfaction problems.
Proc.\ 38th STOC (2006) 130--139.

\bibitem{BGT10}
M.~Bayati, D.~Gamarnik, P.~Tetali:
Combinatorial approach to the interpolation method and scaling limits in sparse random graphs.
Proc.\ 42nd STOC (2010) 105--114.

\bibitem{ACOFriezekSAT}
A.~Coja-Oghlan, A.~Frieze:
Random $k$-SAT: the limiting probability for satisfiability for moderately growing $k$.
Electronic Journal of Combinatorics {\bf15} (2008) N2.

\bibitem{DRZ08}
L.\ Dall'Asta, A.\ Ramezanpour, R.\ Zecchina: Entropy landscape and non-Gibbs solutions in constraint satisfaction problems. Phys. Rev. E 77, 031118 (2008).

\bibitem{Dubois}
O.~Dubois, J.~Mandler: The 3-XORSAT threshold. 
Proc.\ 43rd FOCS (2002) 769--778.

\bibitem{Ehud}
E.~Friedgut:
Sharp thresholds of graph properties, and the {$k$-SAT} problem.
Journal of the AMS {\bf12} (1999) 1017--1054.

\bibitem{EhudHunting}
E.~Friedgut:
Hunting for sharp thresholds.
Random Struct.\ Algorithms {\bf 26} (2005) 37--51

\bibitem{FriezeWormald}
A.~Frieze, N.~Wormald:
Random $k$-Sat: a tight threshold for moderately growing $k$.
Combinatorica {\bf 25} (2005) 297--305.

\bibitem{JLR}
S.~Janson, T.~{\L}uczak, A.~Ruci\'nski: Random Graphs, Wiley  2000.

\bibitem{Kauzmann48}
W.~Kauzmann: The nature of the glassy state and the behavior of liquids at low temperatures.
Chem.\ Rev.\ {\bf 43} (1948)  219--256.


\bibitem{pnas}
F.~Krzakala, A.~Montanari, F.~Ricci-Tersenghi, G.~Semerjian, L.~Zdeborov\'a:
Gibbs states and the set of solutions of random constraint satisfaction problems.
Proc.~National Academy of Sciences {\bf104} (2007) 10318--10323.

%\bibitem{MezardMontanari}
%M.~M\'ezard, A.~Montanari:
%Information, physics, and computation.
%Oxford 2009.

\bibitem{MRT}
A.~Montanari, R.~Restrepo, P.~Tetali:
Reconstruction and clustering in random constraint satisfaction problems.
{\tt arXiv:0904.2751v1} (2009).

\bibitem{MoraZdeb}
T.~Mora, L.~Zdeborov\' a: Random subcubes as a toy model for constraint satisfaction problems. J.\ Stat.\ Phys.\ {\bf131} (2008) 1121--1138.

\bibitem{PittelSorkin}
B.~Pittel, G.~Sorkin:
The satisfiability threshold for $k$-XORSAT.
Preprint (2011).

\bibitem{SPS85} 
Schmidt-Pruzan, J., Shamir, E.: Component structure in the 
evolution of random hypergraphs. Combinatorica {\bf 5} (1985) 
81--94 

%\bibitem{Talagrand}
%M.~Talagrand: Spin glasses: a challenge for mathematicians.
%Springer 2003.



\bibitem{Wormald}
N.~Wormald: The differential equation method for random graph processes and greedy algorithms.
	In M.~Karo\'nski and H.J.\ Pr\"omel (eds.): Lectures on Approximation and randomized algorithms
		(1999) 73--155.

\bibitem{ZdeMez}
L.~Zdeborov\' a, M.~M\'ezard: Constraint satisfaction problems with isolated solutions are hard. J. Stat. Mech. P12004 (2008). 

\end{thebibliography}
\end{document}